\documentclass[12pt]{amsart}
\usepackage[utf8]{inputenc} 
\usepackage[active]{srcltx}
\usepackage{a4wide}
\usepackage{amsthm,amsfonts,amsmath,mathrsfs,amssymb}
\usepackage{dsfont}
\usepackage{mathtools}
\usepackage[T1]{fontenc}
\usepackage[utf8]{inputenc}
\usepackage{enumerate}
\usepackage[left=4cm,top=4cm,right=4cm,bottom=4cm]{geometry}
\usepackage{caption}
\usepackage{subcaption}
\numberwithin{equation}{section}

\usepackage{geometry}
\usepackage{graphicx}
\usepackage{amssymb}
\usepackage{amsmath}
\usepackage{amsthm}
\usepackage{listings}
\usepackage[colorlinks=true,urlcolor=blue,citecolor=blue,linkcolor=blue]{hyperref}
\usepackage{esint}
\usepackage{xcolor}

\usepackage{cite}
\usepackage{etex}
\usepackage[all]{xy}
\usepackage{pgf,tikz}
\usetikzlibrary{arrows}
\usepackage{pgfplots,pgfcalendar}
\usetikzlibrary{patterns}
\usepackage{hyperref}

\def\D{{\mathbb D}}  \def\T{{\mathbb T}}
\def\C{{\mathbb C}}  \def\N{{\mathbb N}}
\def\Z{{\mathbb Z}}

\def\R{\mathbb R}

\def\({\left(}       \def\){\right)}

\newtheorem{theorem}{Theorem}[section]
\newtheorem{lemma}[theorem]{Lemma}
\newtheorem{proposition}[theorem]{Proposition}
\newtheorem{corollary}[theorem]{Corollary}

\theoremstyle{definition}
\newtheorem{definition}[theorem]{Definition}
\newtheorem{example}[theorem]{Example}

\theoremstyle{remark}
\newtheorem{remark}[theorem]{Remark}

\numberwithin{equation}{section}

\DeclareMathOperator*{\esssup}{ess\,sup}
\begin{document}
\title[Average radial integrability]{Average radial integrability spaces of analytic functions}
\author[T. Aguilar-Hern\'andez ]{Tanaus\'u Aguilar-Hern\'andez}
\address{Departamento de Matem\'atica Aplicada II and IMUS, Escuela T\'ecnica Superior de Ingenier\'ia, Universidad de Sevilla,
Camino de los Descubrimientos, s/n 41092, Sevilla, Spain}
\email{taguilar@us.es}

\author[M.D. Contreras]{Manuel D. Contreras}
\email{contreras@us.es}

\author[L. Rodr\'iguez-Piazza]{Luis Rodr\'iguez-Piazza}
\address{Departmento de An\'alisis Matem\'atico and IMUS, Facultad de Matem\'aticas, Universidad
de Sevilla, Calle Tarfia, 41012 Sevilla, Spain}
\email{piazza@us.es}


\date{\today}


\thanks{This research was supported in part by Ministerio de Econom\'{\i}a y Competitividad, Spain,  and the European Union (FEDER), project PGC2018-094215-B-100,  and Junta de Andaluc{\'i}a, FQM133 and FQM-104.}

\maketitle

\begin{abstract}
In this paper we introduce the family of spaces $RM(p,q)$, $1\leq p,q\leq +\infty$. They are spaces of holomorphic functions in the unit disc with average radial integrability. This family  contains the classical Hardy spaces (when $p=\infty$) and Bergman spaces (when $p=q$). We characterize the inclusion between $RM(p_1,q_1)$ and $RM(p_2,q_2)$ depending on the parameters. For $1<p,q<\infty$, our main result provides a characterization of the dual spaces of $RM(p,q)$ by means of the boundedness of the Bergman projection. We show that $RM(p,q)$ is separable if and only if $q<+\infty$. In fact, we provide a method to build isomorphic copies of $\ell^\infty$ in $RM(p,\infty)$.
\end{abstract}

\tableofcontents

\section{Introduction}

In 1923, the classical Hardy spaces $H^p$ were introduced by  F. Riesz \cite{riesz_1923}.   He named those spaces after the article of G.H. Hardy \cite{Hardy_1915}. Subsequently, the Bergman spaces $A^p$ appeared in a work of S. Bergman \cite{bergman_1970} in 1970 focused on the spaces of analytic functions that are square-integrable over a given domain with respect to the Lebesgue area measure.
Since then, great progress has been made in the study of these and other spaces of analytic functions in the unit disc. In most of the cases, the belonging to the space is given in terms of boundedness (or integrability) of a certain average of the function on circles centered at the origin or in terms of the integrability with respect to the Lebesgue area, maybe with a certain weight. There are many good books about these spaces, but we  stand out \cite{duren_theory_2000,garnett_bounded_2007,HKZ,vukotic_multiplier_2016}.

In other less studied cases, the belonging is determined by the average radial integrability. Maybe the most well-known space in this situation is the space of bounded radial variation BRV, a topic that goes back to Zygmund and where many different authors have worked (see, i.e., the papers of Bourgain \cite{Bourgain}, Rudin \cite{rudin_1955}, and Zygmund \cite{zygmund_1944}). The space BRV of analytic functions
with bounded radial variation consists of those holomorphic functions $g\in \mathcal{H}(\D)$ such that
    $$
    \sup_{\theta}\int_0^1|g'(te^{i\theta})|\,dt<\infty.
    $$
    Other different situation where the radial integrability plays an important role is in the Riesz-F\'ejer Theorem which says that there is a constant $C_p>0$ such that if $f$ belongs to the Hardy space $H^{p}$ then
\begin{equation}\label{Riesz-Fejer}
        \sup_{\theta}\left(\int_0^1|f(re^{i\theta})|^{p}\,dr\right)^{1/p}\leq C_p || f ||_{H^{p}}.
\end{equation}
	The left side of \eqref{Riesz-Fejer}, considered as a function in the variables $\theta$ and $r$, is the norm of $f$ in the space $L^\infty(\T,L^p[0,1])$. This paper is devoted to introduce and study the family of spaces $RM(p,q)$ of analytic functions on the disk $\D$ such that $f\in L^{q}(\T,L^p[0,1])$ (Definition~\ref{definition}). This family of spaces contains the Bergman spaces (when $p=q$) and Hardy spaces (when $p=\infty$).

	As far as we know, there is no  systematic study of spaces  of average radial integrability. A second part of this research will appear in \cite{Aguilar-Contreras-Piazza} where Littlewood-Paley type inequalities and integration operators are analyzed in the setting of these spaces. 
	   
   In Section 2, we introduce the family of spaces $RM(p,q)$ and show a range of examples. Among them, we point out Proposition \ref{lacunary} where we characterize lacunary series belonging to $RM(p,q)$. We analyze other properties such as boundedness of evaluation functionals and separability. We show that $RM(p,q)$ is separable if and only if $q<+\infty$ (see Proposition \ref{desnsity-polynomials} and Theorem \ref{non-separability}).  In fact,  $RM(p,\infty)$ always contains a subspace isomorphic to $\ell^{\infty}$ (Theorem~\ref{non-separability}).
   
   The main results of the paper appear in Section 3 and 4. In Section 3 we provide a complete characterization of when one of such spaces is included in another one (Theorem  \ref{thm:inclusion}) and, in such a case, we characterize when the inclusion mapping is compact (Theorem \ref{thm:compactness}). As a byproduct of such characterization,  we see that the converse of \eqref{Riesz-Fejer} does not hold, that is, there are holomorphic functions $f$ in $\D$ such that  $\sup_{\theta}\left(\int_0^1|f(re^{i\theta})|^{p}\,dr\right)^{1/p}<+\infty$ but $f\notin H^{p}$.   

	In the last section of this article we show the boundedness of the Bergman projection from $L^q(\mathbb{T},L^p[0,1])$ onto our spaces $RM(p,q)$ when $1<p,q<\infty$. This allows us to identify the dual space of $RM(p,q)$ for $1<p,q<+\infty$ (Corollary~\ref{dualopenbox}). The proof of the boundedness of the Bergman projection depend on techniques and tools coming form Harmonic Analysis. In particular, we use a classical result of C. Fefferman and E. Stein.
	The case $p=q$ gives the well-known  boundedness of the Bergman projection from $L^p(\D)$ onto the Bergman space $A^p$, which is usually proved with different techniques not working in our situation. 

%
%
%

Throughout the paper the letter $C=C(\cdot)$ will denote an absolute constant whose value depends on the parameters indicated in the parenthesis, and may change from one occurrence to another. We will use the notation $a\lesssim b$ if there exists a constant $C=C(\cdot)>0$ such that $a\leq C b$, and $a\gtrsim b$ is understood in an analogous manner. In particular, if $a\lesssim b$ and $a\gtrsim b$, then we will write $a\asymp b$.

\section{Definition and first properties}

We start this section introducing the spaces which are the goals of our study and providing some different kind of functions that belongs to them. In addition, we deal with some properties of such spaces, as for instance the separability. 

\begin{definition} \label{definition}
	Let $0< p,q \leq +\infty$. We define the spaces of analytic functions
	$$
		RM(p,q)=\{f\in\mathcal{H}(\D)\ :\rho_{p,q}(f)<+\infty\}
	$$
	where
	\begin{equation*}
	\begin{split}
	\rho_ {p,q}(f)&=\left(\frac{1}{2\pi}\int_{0}^{2\pi} \left(\int_{0}^{1} |f(r e^{i t})|^p \ dr \right)^{q/p}dt\right)^{1/q}, \quad \text{ if } p,q<+\infty,\\
	\rho_ {p,\infty}(f)&=\esssup_ {t\in[0,2\pi)}\left(\int_{0}^{1} |f(r e^{i t})|^p \ dr \right)^{1/p}, \quad \text{ if } p<+\infty, \\
	\rho_ {\infty,q}(f)&=\left(\frac{1}{2\pi}\int_{0}^{2\pi} \left(\sup_{r\in [0,1)} |f(r e^{i t})| \right)^{q}dt\right)^{1/q},\quad\text{ if } q<+\infty,\\
	\rho_{\infty,\infty}(f)&=\|f\|_{H^{\infty}}.
	\end{split}
	\end{equation*}	
\end{definition}

\begin{remark}\label{remarkesssupremum}
	In the definition of $\rho_{p,\infty}$ the essential supremum can be replaced by the supremum. Fix $\theta\in [0,2\pi]$. As the set of $t\in [0,2\pi]$ such that $\left(\int_{0}^{1} |f(re^{it})|^{p}\ dr\right)^{1/p}\leq \rho_{p,\infty}(f)$ is dense in $[0,2\pi]$, we can extract a sequence $\{t_n\}$ in this set such that $t_n\rightarrow \theta$. Using Fatou's lemma it follows that $\left(\int_{0}^{1} |f(re^{i\theta})|^{p}\ dr\right)^{1/p}\leq \rho_{p,\infty}(f)$.
\end{remark}

One can easily check that if $1\leq p,q \leq +\infty$, then $RM(p,q)$ is a Banach space when we endow it with the norm $\rho_{p,q}$. In fact in this paper, we will be interested only in these cases. So, we will stand most of our results for $1\leq p,q\leq +\infty$. Nevertheless, sometimes in the proofs of such results considering other values of $p$ and $q$ will help us.


For certain parameters $p,q$ these spaces $RM(p,q)$ are well known spaces. Namely, it is clear that $RM(p,p)$ is nothing but the Bergman space $A^p$, for $1\leq p<\infty$. In addition, let us fix $1\leq q\leq +\infty$. One can check that $RM(\infty, q) $ is contained in the Hardy space $H^q$. On the other hand, by \cite[Theorem 17.11(a), p. 340]{rudin_real_1987}, there is a constant $C=C(q)$ such that if $f\in H^q$, then
\begin{equation*}
\begin{split}
\int_0^{2\pi} \sup_{r\in [0,1]} |f(re^{i\theta})|^q \, d\theta&\leq \int_0^{2\pi}  \sup\{ |f(z)|^q: \, |e^{i\theta}-z|<3(1-|z|)\} \, d\theta\leq C\Vert f\Vert ^q_{H^q},
\end{split}
\end{equation*}
so that we get that $RM(\infty, q) =H^q$ for all $q\in (0,+\infty]$. 
Another interesting space that fits in this family is the space of bounded radial variation $BRV$ (see, i.e. \cite{Bourgain}), that is the space of analytic functions such that $f'\in RM(1,\infty)$.

\subsection{First examples}

\begin{example}\label{ex1}
For $\alpha\in\R$ and $0<p,q\leq +\infty$, the function $f_{\alpha}(z)=(1-z)^{-\alpha}$ belongs to $RM(p,q)$ if and only if $\alpha<\frac{1}{p}+\frac{1}{q}$. 
\end{example}
\begin{proof} If $\alpha\leq 0$, then the function $|f_{\alpha}|$ is bounded so that it belongs to $RM(p,q)$ for all $p$ and $q$. Thus, in what follows we will only consider the case $\alpha>0$. 

Assume now that $0< p,q<\infty$. Write $I(t)=\int_0^1|f_\alpha (re^{it})|^p\, dr$. Since $I$ is even and decreasing in $[0,\pi]$, we have  
$$
\int_{0}^{\pi/4}I(t)^{q/p}\, dt\leq \rho_{p,q}^q(f_\alpha)=2\int_{0}^{\pi}I(t)^{q/p}\, dt\leq 8\int_{0}^{\pi/4}I(t)^{q/p}\, dt.
$$ 
In addition, for $t\in [0,\pi/4]$, we have that $1-\cos(t)\asymp t^2/2$. Therefore,
$$
\rho_{p,q}^q(f_\alpha)\asymp\int_{0}^{\pi/4}\left[\int_0^1 \frac{1}{((1-r)^2+rt^2)^{\alpha p/2}}\, dr \right]^{q/p}\, dt.
$$ 
With a similar argument, we can reduce the integral in $r$ to the interval $[1/2,1]$ and using that when $r$  runs this interval, the function $rt^2$ is equivalent to $t^2$ we have
\begin{equation}\label{Eq:ex1}
\rho_{p,q}^q(f_\alpha)\asymp\int_{0}^{\pi/4}\left[\int_{1/2}^1 \frac{1}{((1-r)^2+t^2)^{\alpha p/2}}\, dr \right]^{q/p}\, dt.
\end{equation}
If $\alpha\geq \frac{1}{p}+\frac{1}{q}$, and $t\in [0,1/2]$, then
$$
\int_{1/2}^1 \frac{1}{((1-r)^2+t^2)^{\alpha p/2}}\, dr\geq \int_{1-t}^1 \frac{1}{((1-r)^2+t^2)^{\alpha p/2}}\, dr\geq \int_{1-t}^1 \frac{1}{(2t^2)^{\alpha p/2}}\, dr=\frac{1}{2^{\alpha p/2}}\frac{1}{t^{\alpha p-1}}.
$$
Thus
$$
\int_{0}^{\pi/4}\left[\int_{1/2}^1 \frac{1}{((1-r)^2+t^2)^{\alpha p/2}}\, dr \right]^{q/p}\, dt\geq \frac{1}{2^{\alpha q/2}}
\int_{0}^{1/2}\left[ \frac{1}{t^{\alpha p-1}}\right]^{q/p}\, dt=+\infty.
$$ 
and so, by \eqref{Eq:ex1}, $f_\alpha$ does not belong to $RM(p,q)$. 

If $\alpha p<1$, then 
$$
\int_{1/2}^1 \frac{1}{((1-r)^2+t^2)^{\alpha p/2}}\, dr \leq \int_{1/2}^1 \frac{1}{(1-r)^{\alpha p}}\, dr<+\infty,
$$
so that, by \eqref{Eq:ex1}, $f_\alpha\in RM(p,q)$. If $\alpha p=1$, then we obtain
	\begin{align*}
	\int_{1/2}^{1} \frac{1}{((1-r)^2+t^2)^\frac{\alpha p}{2}}\ dr\leq \int_{1-t}^{1} \frac{1}{t}\ dr+\int_{1/2}^{1-t} \frac{1}{1-r}\ dr\leq \ln\left(\frac{e}{2t}\right).
	\end{align*}
	Integrating with respect to $t$ it follows 
	\begin{align*}
	\int_{0}^{\pi/4} \left(\int_{1/2}^{1} \frac{1}{((1-r)^2+t^2)^\frac{\alpha p}{2}}\ dr\right)^{q/p}\ dt\leq \int_{0}^{\pi/4}\ln^{q/p}\left(\frac{e}{2t}\right) dt<+\infty.
	\end{align*}

It remains to see what happens if $1< \alpha p<1+\frac{p}{q}.$ In this case, if $t\in [0,\pi/4]$, we have
$$
\int_{1/2}^1 \frac{1}{((1-r)^2+t^2)^{\alpha p/2}}\, dr\leq \int_{1/2}^{1-t} \frac{1}{(1-r)^{\alpha p}}\, dr +\int_{1-t}^1 \frac{1}{t^{\alpha p}}\, dr \leq \frac{\alpha p}{\alpha p-1} \frac{1}{t^{\alpha p-1}}.
$$
 Therefore, by  \eqref{Eq:ex1},
$$
\rho_{p,q}^q(f_\alpha)\lesssim\left(\frac{\alpha p}{\alpha p-1}\right)^{q/p}\int_0^{\pi/4} \frac{1}{t^{\alpha q-\frac{q}{p}}}\, dt<+\infty.
$$

Summing up, the result holds if both $p$ and $q$ are finite. For $p=\infty$, since $RM(\infty,q)=H^q$, the result is well-known (see, i.e., \cite[Page 13]{duren_theory_2000}).

For $q=\infty$, arguing as above we have 
\begin{equation}\label{Eq:ex2}
\rho_{p,\infty}^p(f_\alpha)\asymp\sup_{0\leq t\leq \pi/2}\int_{1/2}^1 \frac{1}{((1-r)^2+t^2)^{\alpha p/2}}\, dr =\int_{1/2}^{1} \frac{dr}{(1-r)^{\alpha p}}<+\infty
\end{equation}
if and only if $\alpha<\frac{1}{p}$.
\end{proof}

\begin{example}\label{ex3} Let $1\leq p,q<\infty$, $n\geq 1$ and take $\alpha$ such that $\frac{1}{p}+\frac{1}{q}=\frac{1}{\alpha}$. The $RM(p,q)$-norm of the holomorphic function $$f_{n,\alpha}(z)=\left(\sum_{k=0}^{n}z^k\right)^{1/\alpha}=\left(\frac{1-z^{n+1}}{1-z}\right)^{1/\alpha},$$ where we are using the main branch of the logarithm to define $w^{1/\alpha}$, can be estimated as
	\begin{align}\label{estimatesequencefunctions}
	\rho_{p,q}(f_{n,\alpha})\lesssim \left(\frac{p}{p-\alpha}\right)^{1/p}\ln^{1/q}(n+1).
	\end{align}
\end{example}
\begin{proof}
Clearly, $f_{n,\alpha}$ is well-defined. It is not difficult to see that the proof of \eqref{estimatesequencefunctions} can be reduced to the case $\alpha=1$ and $\frac{1}{p}+\frac{1}{q}=1$. Notice that, in this case,  $p>1$. 

Since 
$$
\int_{\pi/4}^{\pi}\left(\int_{0}^{1}|f_{n,1}(re^{i\theta})|^{p}\, dr\right)^{q/p}\, d\theta \leq \int_{\pi/4}^{\pi}\left(\int_{0}^{1}|4|^{p}\, dr\right)^{q/p}\, d\theta =\frac{3\pi}{4}4^{q},
$$
we have 
\begin{align*}
2\pi \rho_{p,q}(f_{n,1})^{q}\leq 2 \int_{0}^{\pi/4}  \left(\int_{0}^{1} |f_{n,1}(re^{i\theta})|^p\ dr\right)^{q/p}\ d\theta +\frac{3\pi}{4}4^{q}.
\end{align*}

If $1-\theta\leq r\leq 1$ and $\theta\in[0,\pi/4]$, arguing as in Example \ref{ex1}, we obtain 
$$|f_{n,1}(re^{i\theta})|\lesssim  \frac{2}{\sqrt{(1-r)^2+r \theta^2}}\leq \frac{2}{\theta\sqrt{1-\theta}}\leq \frac{2}{\theta\sqrt{1-\frac{\pi}{4}}}<\frac{5}{\theta}.$$ 
Therefore, there is a constant $C>0$ such that
\begin{equation*}
\begin{split}
	&\int_{\frac{1}{n+1}}^{\pi/4}\left(\int_{0}^{1} |f_{n,1}(re^{i\theta})|^p\ dr\right)^{q/p}\ d\theta=\\
	&\qquad =\int_{\frac{1}{n+1}}^{\pi/4}\left(\int_{0}^{1-\theta}|f_{n,1}(re^{i\theta})|^p\ dr+ \int_{1-\theta}^{1}|f_{n,1}(re^{i\theta})|^p\ dr\right)^{q/p}\ d\theta\\
	&\qquad\leq \int_{\frac{1}{n+1}}^{\pi/4}\left(\int_{0}^{1-\theta}\frac{2^p}{(1-r)^p}\ dr+ \int_{1-\theta}^{1}\frac{C^p}{\theta^p}\ dr\right)^{q/p}\ d\theta \\
	&\qquad \leq C^{q}\int_{\frac{1}{n+1}}^{\pi/4}\left( \frac{p}{p-1}\theta^{-p+1}\right)^{q/p}\ d\theta=C^{q}\left(\frac{p}{p-1}\right)^{q/p}(\ln (\pi/4)+\ln (n+1))
\end{split}
\end{equation*}
and
\begin{equation*}
\begin{split}
	&\int_{0}^{\frac{1}{n+1}}\left(\int_{0}^{1} |f_{n,1}(re^{i\theta})|^p\ dr\right)^{q/p}\ d\theta=\\
	&\qquad =\int_{0}^{\frac{1}{n+1}}\left(\int_{0}^{1-\frac{1}{n+1}}|f_{n,1}(re^{i\theta})|^p\ dr+\int_{1-\frac{1}{n+1}}^{1}|f_{n,1}(re^{i\theta})|^p\ dr\right)^{q/p}\ d\theta\\
	&\qquad \leq 2^{q}\int_{0}^{\frac{1}{n+1}}\left(\int_{0}^{1-\frac{1}{n+1}}\frac{1}{(1-r)^p}\ dr+\int_{1-\frac{1}{n+1}}^{1}(n+1)^p \ dr\right)^{q/p}\ d\theta \leq 2^{q}\left(\frac{p}{p-1}\right)^{q/p}.
\end{split}
\end{equation*}
Putting altogether, we get the estimation of $\rho_{p,q}(f_{n,\alpha})$.
\end{proof}

Next example provides the lacunary series that belong to $RM(p,q)$. For $p=\infty$, that is for Hardy spaces, the characterization is different and it can be seen in \cite[Theorem 6.2.2]{vukotic_multiplier_2016} for $q<+\infty$ and in \cite[Vol. I, p. 247]{zygmund_1959} for $q=\infty$. We will say that a sequence of positive numbers $\{x_k\}$ is a lacunary sequence if there is a constant $\lambda$ such that $\frac{x_{k+1}}{x_k}\geq \lambda>1$.

\begin{proposition}\label{lacunary}
Let $\{n_k\}_{k=0}^{\infty}$ be a lacunary sequence of positive integer numbers, $1\leq p<\infty$ and $1\leq q \leq \infty$. Then 
\begin{align*}
f(z)=\sum_{k=0}^{\infty} \alpha_{k} z^{n_k}
\end{align*}
belongs to $RM(p,q)$ if and only if
\begin{align*}
\sum_{k=0}^{\infty} \frac{|\alpha_k|^p}{n_k}<+\infty. 
\end{align*}
Moreover, it is satisfied that 
\begin{align}\label{equation-lacunary}
\rho_{p,q}(f)\asymp \left(\sum_{k=0}^{\infty} \frac{|a_k|^p}{n_k}\right)^{1/p}.
\end{align}
\end{proposition}

\begin{remark}
	Notice that the second part of the expression \eqref{equation-lacunary} does not depend on $q$.
\end{remark}

\begin{proof} Notice that $\left\{n_k+\frac{1}{p}\right\}_{k\geq 0}$ is also a lacunary sequence. The proof of this result is based on a characterization of bases on $L^{p }[0,1]$ due to Gurari\v{\i}  and Macaev \cite{gurariui_lacunary_1966}. Namely they proved that, fixed $p\in [1,+\infty)$, if a sequence $\{n_{k}\}_{k\geq 0}$ is lacunary then there exist two positive constants $A$ and $B$ such that
\begin{equation}\label{Eq:lacunary}
A\left( \sum_{k=0}^{\infty}|\beta_{k}|^{p }\right)^{1/p}\leq \left\Vert \sum_{k=0}^{\infty} \beta_{k}\sqrt[p]{n_{k}+1/p} \, 
t^{n_{k}} \right\Vert _{L^{p}}\leq B\left( \sum_{k=0}^{\infty}|\beta_{k}|^{p }\right)^{1/p},
\end{equation}
for every $\{\beta_k\}\in \ell^p$.

Take now $f(z)=\sum_{k=0}^{\infty}\alpha_{k}z^{n_{k}}$ a holomorphic function in the unit disc. Fix $\theta\in [0,2\pi]$ and write $\beta_{k}:=\frac{\alpha_{k}}{\sqrt[p]{n_{j}+1/p}}e^{i\theta n_{k} }$ if $k\geq 0$.
By \eqref{Eq:lacunary},
\begin{equation*}\label{Eq:lacunary2}
\begin{split}
A\left( \sum_{k=0}^{\infty}\frac{|\alpha_{k}|^{p }}{n_{k}+1/p}\right)^{1/p}&\leq \left\Vert \sum_{k=0}^{\infty} \alpha_{k} \, 
r^{n_{k}} e^{i\theta n_{k}}\right\Vert _{L^{p}}=\left(\int_{0}^{1}|f(re^{i\theta})|^{p}\right)^{1/p}\\
&\leq B\left( \sum_{k=0}^{\infty}\frac{|\alpha_{k}|^{p }}{n_{k}+1/p}\right)^{1/p}.
\end{split}
\end{equation*}
Now, looking at the very definition of $\rho_{p,q}$ we get the result. 
\end{proof}
%

\subsection{Evaluating functionals} This subsection is devoted to the functionals $f\mapsto f(z)$ and $f\mapsto f'(z)$. We prove that both of them are bounded and estimate their norms. We will need the following inclusion. 

\begin{proposition}\label{Hardyineq}
Let $0< s\leq+\infty$. Then
$H^{s}\subset RM(p,q)$ if and only if  $\frac{1}{p}+\frac{1}{q}\geq \frac{1}{s}$.
\end{proposition}
\begin{proof} The result is clear for $s=+\infty $ since $H^\infty$ is a subspace $RM(p,q)$ for every $p,q$. Thus, from now on we consider the case $s<+\infty$. 
Assume that $H^s\subset RM(p,q)$ and suppose that $\frac{1}{s}>\frac{1}{p}+\frac{1}{q}$. Take $\frac{1}{s}>\alpha>\frac{1}{p}+\frac{1}{q}$. Then the function $f_\alpha$ defined in Example~\ref{ex1} belongs to $RM(\infty,s)=H^s$ and not to $RM(p,q)$. A contradiction. Thus  $\frac{1}{p}+\frac{1}{q}\geq \frac{1}{s}$.

To see that converse implication we claim that $H^{s}\subset RM(p_1,q_1)$ whenever $1\leq s<+\infty$ and $\frac{1}{p_1}+\frac{1}{q_1}= \frac{1}{s}$. Assume for the moment that the claim holds.  Fix $p$ and $q$ such that $\frac{1}{p}+\frac{1}{q}\geq \frac{1}{s}$. We consider $p_1\geq p$ and $q_1\geq q$ such as $\frac{1}{p_1}+\frac{1}{q_1}=\frac{1}{s}$. By the claim  $H^s\subset RM(p_1,q_1)$. Moreover, it is easy prove that $RM(p_1,q_1)\subset RM(p,q)$ using H\"older's inequality twice for $p_1\geq p$ and $q_1\geq q$. 

Thus it remains to prove the claim. 
	By F\' ejer-Riesz theorem \cite[Theorem 3.13, p. 46]{duren_theory_2000}, we have that for each $f\in H^{s}$ and $\theta$,
	$$
	\left(\int_{0}^{1}|f(re^{i\theta})|^{s}\, dr\right)^{1/s}\lesssim \Vert f\Vert_{H^{s}}
	$$ 
(notice that, in particular, this implies that $H^s\subset RM(s,\infty)$). 
	Now, since $\frac{1}{p_1}+\frac{1}{q_1}= \frac{1}{s}$ we can take $\lambda\in [0,1]$ such that $\frac{1}{p_1}=\frac{\lambda}{s}$ and $\frac{1}{q_1}=\frac{1-\lambda}{s}$. Then
	\begin{align*}
	\rho_{p_1,q_1}(f)&=\left(\int_{0}^{2\pi} \left(\int_{0}^{1}|f(re^{i\theta})|^{p_1(1-\lambda)}|f(re^{i\theta})|^{p_1\lambda}\ dr\right)^{q_1/p_1}\ \frac{d\theta}{2\pi}\right)^{1/q_1}\\
	&\leq \left(\int_{0}^{2\pi} \sup_{r} |f(re^{i\theta})|^{q_1(1-\lambda)} \left(\int_{0}^{1}|f(re^{i\theta})|^{s}\ dr\right)^{q_1\lambda/s}\ \frac{d\theta}{2\pi}\right)^{1/q_1}\\
	&\leq \left(\int_{0}^{2\pi} \sup_{r} |f(re^{i\theta})|^{s} \left(\int_{0}^{1}|f(re^{i\theta})|^{s}\ dr\right)^{q_1\lambda/s}\ \frac{d\theta}{2\pi}\right)^{1/q_1}\\
	&\lesssim \Vert f\Vert_{H^{s}}^\lambda \left(\int_{0}^{2\pi} \sup_{r} |f(re^{i\theta})|^{s} \frac{d\theta}{2\pi}\right)^{1/q_1}=\Vert f\Vert_{H^{s}}^\lambda\rho_{\infty,s}(f)^{1-\lambda}<+\infty.
	\end{align*} 
	Hence, we have proved the claim and we are done.
	\end{proof}

\begin{proposition}\label{estimevfunc}
	Let $0< p,q \leq\infty$.  If $z\in \D$, then the functional $\delta_z:RM(p,q)\to \C$ given by $\delta_z(f):=f(z)$, for all $f\in RM(p,q)$, is continuous and 
	\begin{align*}
	\| \delta_{z}\|_{(RM(p,q))^{\ast}} \asymp \frac{1}{(1-|z|^2)^{\frac{1}{p}+\frac{1}{q}}}.
	\end{align*}
\end{proposition}
\begin{proof}  The subharmonicity of the function $|f|^{p_{0}}$ shows that for all $z\in \D$, 
\begin{align*}
|f(z)|^{p_{0}}\leq \frac{1}{\pi r^2} \int_{B(z,r)} |f(w)|^{p_{0}}\ dA(w)
\end{align*}
where $r=1-|z|$, $B(z,r)$ is the disc centered at $z$ with radius $r$ and $dA(w)$ means integration with respect to the Lebesgue measure on the unit disc $\D$.

Due to the rotational invariance of the space $RM(p,q)$ we can assume that $z$ belongs to the interval $[0,1)$. Take $f\in RM(p,q)$.
Fix  $p_0>0$. To prove the result we may assume that $\frac{1}{2}\leq z<1$.
Bearing in mind  that $$\arcsin\left(\frac{1-z}{z}\right) \leq \pi (1-z)$$ for $\frac{1}{2}\leq z<1$, we have $|{\rm Arg}(w)|\leq \pi r$ for $w\in B(z,r)$. It follows 
	\begin{align*}
	\frac{1}{\pi r^2} \int_{B(z,r)} |f(w)|^{p_{0}}\ dA(w) \leq \frac{1}{\pi r^2} \int_{-\pi r}^{\pi r} \left(\int_{1-2r}^{1} |f(\rho e^{i\theta})|^{p_{0}}\ d\rho\right)\ d\theta.
	\end{align*}
	If $p,q\geq p_{0}$, H\"older's Inequality twice gives
	\begin{align*}
	&\frac{1}{\pi r^2} \int_{-\pi r}^{\pi r} \left(\int_{1-2r}^{1} |f(\rho e^{i\theta})|^{p_{0}}\ d\rho\right)\ d\theta\leq \frac{2}{ r^2} \int_{-\pi r}^{\pi r} \left(\int_{1-2r}^{1} |f(\rho e^{i\theta})|^p \ d\rho \right)^{\frac{p_0}{p}} (2r)^{1-\frac{p_0}{p}}\ \frac{d\theta}{2\pi}\\
	&\leq \frac{2^{2-\frac{p_0}{p}}}{ r^{1+\frac{p_0}{p}}} \left(\int_{-\pi r}^{\pi r} \left(\int_{1-2r}^{1} |f(\rho e^{i\theta})|^p \ d\rho \right)^{\frac{q}{p}}\ \frac{d\theta}{2\pi}\right)^{\frac{p_0}{q}}  r^{1-\frac{p_0}{q}}\leq  \frac{2^{2-\frac{p_0}{p}}}{(1-|z|)^{p_{0}\left(\frac{1}{p}+\frac{1}{q}\right)}}\rho_{p,q}^{p_{0}}(f).
	\end{align*}
	So that
	\begin{align}\label{eqfung}
	|f(z)|^{p_{0}}\leq \frac{2^{2-\frac{p_0}{p}}}{(1-|z|)^{\frac{^{p_{0}}}{p}+\frac{^{p_{0}}}{q}}} \rho_{p,q}(f)^{p_{0}}.
	\end{align}
	Hence $\delta_z$ is continuous and 	
	$
	\|\delta_{z}\|\lesssim 1/{(1-|z|)^{\frac{1}{p}+\frac{1}{q}}}.
	$

To see the converse inequality, take $r$ such that $\frac{1}{p}+\frac{1}{q}=\frac{1}{r}$. Using constant functions we see the converse if $r=+\infty$. In the remaining cases, by Proposition \ref{Hardyineq}, $RM(\infty,r)=H^{r}\subset RM(p,q)$ and thus
	\begin{align*}
	\|\delta_{z}\|_{(RM(p,q))^{\ast}}\gtrsim \|\delta_{z}\|_{(RM(\infty,r))^{\ast}}\asymp \|\delta_z\|_{(H^r)^{\ast}}=\frac{1}{(1-|z|^2)^{1/r}},
	\end{align*}
where we have used \cite[Exercise 2, p. 86]{vukotic_multiplier_2016} or \cite[Exercise 5, p. 85]{garnett_bounded_2007}. 
\end{proof}

\begin{proposition}\label{estiderfunc}
	Let $1\leq  p,q \leq\infty$.  If $z\in \D$, then the functional $\delta'_z:RM(p,q)\to \C$ given by $\delta'_z(f):=f'(z)$, for all $f\in RM(p,q)$, is continuous and 
	\begin{align*}
	\| \delta'_{z}\|_{(RM(p,q))^{\ast}}\asymp \frac{1}{(1-|z|^2)^{\frac{1}{p}+\frac{1}{q}+1}}.
	\end{align*}
\end{proposition}

\begin{proof} Again we assume that $z\in[0,1)$. Fix $z\in [0,1)$ and denote by $C$ the boundary of the disc centered at $z$ and with radius $(1-|z|)/2$. The Cauchy's integral formula  and the estimate of the evaluation functional given in Proposition~\ref{estimevfunc} show 
	\begin{align*}
	|f'(z)|&\leq \frac{1}{\pi}\int_{0}^{2\pi} \frac{|f(z+\frac{1-z}{2}e^{i\theta})|}{(1-z)}\ d\theta\lesssim \frac{1}{\pi(1-z)}\int_{0}^{2\pi} \frac{\rho_{p,q}(f)}{(1-|z+\frac{1-z}{2}e^{i\theta}|)^{\frac{1}{p}+\frac{1}{q}}}\ d\theta\\
	&\leq \frac{2^{\frac{1}{p}+\frac{1}{q}+1}}{(1-z)^{\frac{1}{p}+\frac{1}{q}+1}} \rho_{p,q}(f)\lesssim \frac{\rho_{p,q}(f)}{(1-z^2)^{\frac{1}{p}+\frac{1}{q}+1}} . 
	\end{align*}

	To prove the converse inequality, we will use a similar argument to the one given in  Proposition \ref{estimevfunc}. Using Proposition~\ref{Hardyineq} we have that $RM(\infty,r)=H^{r}\subset RM(p,q)$ for $\frac{1}{p}+\frac{1}{q}=\frac{1}{r}$. So, it follows 
	\begin{align*}
	\|\delta'_{z}\|_{(RM(p,q))^{\ast}}\gtrsim \|\delta'_{z}\|_{(RM(\infty,r))^{\ast}}\asymp \|\delta'_z\|_{(H^r)^{\ast}}.
	\end{align*}
On the one hand, if $r<+\infty$, since for the Hardy space $H^r$ it is known that $\|\delta'_{z}\|_{(H^r)^{\ast}}\asymp \frac{1}{(1-|z|^2)^{\frac{1}{r}+1}}$   \cite[Exercise 5, p. 85]{garnett_bounded_2007}, we obtain 
	\begin{align*}
	\|\delta'_{z}\|_{(RM(p,q))^{\ast}}\gtrsim \frac{1}{(1-|z|^2)^{\frac{1}{r}+1}}.
	\end{align*}
On the other hand, if $r=+\infty$, take the function $\varphi(w):=\frac{w-z}{1-\overline z w}$, $w\in \D$. Since $\varphi$ is an automorphism of the unit disc, we have that $||\varphi||_{{H^{\infty}}}=1$ and 
$$
||\delta'_{z}||_{{(H^{\infty})^{*}}}\geq |\varphi'(z)|=\frac{1}{1-|z|^{2}}.
$$
And we end with a similar argument.
\end{proof}

Combining Propositions \ref{estimevfunc} and \ref{estiderfunc}, the next corollary follows.

\begin{corollary}\label{evderestim}
	Let $1\leq p, q \leq+\infty$.  If $z\in \D$, then
	\begin{align*}
	\|\delta'_{z}\|_{(RM(p,q))^{\ast}}\asymp \frac{\|\delta_{z}\|_{(RM(p,q))^{\ast}}}{1-|z|^2}.
	\end{align*}
\end{corollary}

\subsection{Density of polynomials and separability}

Next lemma is obvious if $f$ is continuous (and then uniformly continuous) and by density of such functions we extend to the whole space:

\begin{lemma}\label{p2}
	Let $f\in L^{p}([0,1])$, $1\leq p<+\infty$. Then 
	\begin{align}\label{eq3}
	\lim_{\rho\to 1}\int_{0}^{1}|f(x)-f(\rho x)|^{p}\ dx =0.
	\end{align}
\end{lemma}

Given a holomorphic function $f$ in the unit disc and $0<r<1$, we define $f_{r}(z):=f(rz)$, for all $z\in \D$.
\begin{proposition}\label{p3}
Let $1\leq p\leq +\infty$, $1\leq q<+\infty$.	If $f\in RM(p,q)$, then $\rho_{p,q}(f-f_{r})\rightarrow 0$ when $r\rightarrow 1^{-}$.
\end{proposition}
\begin{proof} Assume first that $p$ is finite. We define
	\begin{align*}
	R_{p}(\theta,f)=\left(\int_{0}^{1} |f(ue^{\theta i})|^{p}\ du\right)^{1/p},
	\end{align*}
	what it is well-defined for almost every $\theta$.
	Easily we can see that $R_{p}(\theta,f-f_r)\leq R_{p}(\theta,f_r)+R_{p}(\theta,f)$. Now we consider $r>1/2$, then we have 
	\begin{align*}
	R_{p}(\theta,f_r)^{p}=\int_{0}^{1} |f(ru e^{\theta i})|^{p}\ du =\int_{0}^{r} |f(ue^{\theta i})|^{p} \frac{du}{r}\leq \frac{1}{r} R_{p}(\theta,f)^{p}<2R_{p}(\theta,f)^{p}.
	\end{align*}
	Hence, we have that $R_{p}(\theta,f-f_r)\leq 3 R_{p}(\theta,f)$. By Lemma~\ref{p2},   $R_{p}(\theta,f-f_r)\rightarrow 0$, when $r\rightarrow 1$. Since the function $[0,2\pi]\ni \theta\mapsto R_{p}(\theta,f)$ is integrable, using the dominated convergence theorem we conclude the proof for $p<+\infty$. For $H^p$ spaces, this result is known \cite[Theorem 2.6, p. 21]{duren_theory_2000}.
\end{proof}

\begin{proposition}\label{desnsity-polynomials}
Let $1\leq p\leq +\infty$, $1\leq q<+\infty$.	Polynomials are dense in $RM(p,q)$. In particular, $RM(p,q)$ is a separable space.
\end{proposition}
\begin{proof}
	We will study the cases $1\leq p,q<\infty$ since it is well-known that polynomials are dense in Hardy spaces $H^q=RM(\infty,q)$ for $0<q<\infty$. Let $f\in RM(p,q)$. Let us fix $r<1$. The function $f_r$ is holomorphic on $\frac{1}{r}\D$. Since $\overline{\D}\subset  \frac{1}{r}\D$ with $r\in (0,1)$, the sequence of partial sums $\{P_{n}\}_n$ of the Taylor expansion of $f_r$ converges uniformly to $f_r$ in $\overline{\D}$. Therefore, polynomials $\{P_n\}_n$ converges in the topology of the $RM(p,q)$-norm to $f_r$ and together with Proposition~\ref{p3} we obtain that polynomials are dense in $RM(p,q)$.
\end{proof}

It is well-known that $H^\infty=RM(\infty,\infty)$ is a non-separable Banach space. 
In order to study the non-separability of  $RM(p,\infty)$, for $p<+\infty$, we introduce:
 
\begin{definition}
	Let $1\leq p< +\infty$. We define the subspace $RM(p,0)$ of $RM(p,\infty)$
	\begin{align*}
	RM(p,0):=\left\{ f\in \mathcal{H}(\D):\lim\limits_{\rho\rightarrow 1} \sup_{\theta} \left(\int_{\rho}^{1} |f(re^{i\theta})|^{p} dr\right)^{1/p}=0\right\}.
	\end{align*} 
\end{definition}

It can be proved that $RM(p,0)$ is a closed subspace of $RM(p,\infty)$, so that it is a Banach space. We will show later that $RM(p,\infty)\neq RM(p,0)$.

Now, we can provide an analogous to Proposition \ref{p3} for $q=\infty$:

\begin{proposition}\label{propp0unif}
	Let $1\leq p <+\infty$ and $f\in RM(p,\infty)$. Then $f_r\in RM(p,0)$. Moreover, $f\in RM(p,0)$ if and only if 
\begin{equation}\label{Eq:propp0unif}
\rho_{p,\infty}(f-f_{r})\rightarrow 0
\end{equation} when $r\rightarrow 1$.
\end{proposition}
\begin{proof}
Assume that $f\in RM(p,\infty)$. Notice that, for $\rho<1$,
\begin{align*}
	\sup_{\theta} \left(\int_{\rho}^{1} |f_{r}(ue^{\theta i})|^{p}\ du\right)^{1/p}=\sup_{\theta} \left(\int_{r\rho}^{r} |f(ue^{\theta i})|^{p}\ \frac{du}{r}\right)^{1/p}\leq (1-\rho)\sup_{z\in \D} |f_{r}(z)|.
	\end{align*}
Thus
\begin{align*}
	\lim\limits_{\rho\rightarrow 1}\sup_{\theta} \left(\int_{\rho}^{1} |f_{r}(ue^{\theta i})|^{p}\ du\right)^{1/p}=0.
	\end{align*}
	This implies that $f_{r}\in RM(p,0)$. Since $RM(p,0)$ is closed in $RM(p,\infty)$, we get $f\in RM(p,0)$ if \eqref{Eq:propp0unif} holds.

Assume now that $f\in RM(p,0)$, we have to see that $\rho_{p,\infty}(f-f_r)\rightarrow 0$. Fix $\varepsilon>0$. Then there is $\rho_{0}<1$ such that
\begin{equation}\label{Eq:propp0unif2}
\sup_{\theta} \left(\int_{\rho}^{1} |f(se^{i\theta})|^{p}\,  ds\right)^{1/p}\leq \varepsilon
\end{equation}
for all $\rho_{0}\leq \rho<1$.
Take $\rho= (\rho_{0}+1)/2$ and $r<1$ such that $r\rho>\rho_{0}$. A compactness argument shows that
$$
\lim_{r\to 1} \sup_{\theta}\int_{0}^{\rho} |f(se^{\theta i})-f_{r}(s e^{\theta i})|^{p}\ ds=0.
$$
 Bearing in mind \eqref{Eq:propp0unif2}, for each $\theta$, we have
\begin{equation*}
\begin{split}
 \left(\int_{\rho}^{1} |f(se^{\theta i})-f_{r}(s e^{\theta i})|^{p}\,  ds\right)^{1/p}&\leq \left(\int_{\rho}^{1} |f(se^{\theta i})|^{p}\,  ds\right)^{1/p}+ \left(\int_{\rho}^{1} |f_{r}(s e^{\theta i})|^{p}\,  ds\right)^{1/p}\\
 &\leq \varepsilon+ \frac{1}{r}\left(\int_{r\rho}^{r} |f(s e^{\theta i})|^{p}\,  ds\right)^{1/p}\leq \varepsilon+\frac{1}{r}\varepsilon.
 \end{split}
\end{equation*}
This implies that $\limsup_{r\to 1}\rho_{p,\infty}(f-f_{r})\leq 2\varepsilon$. Thus $\lim_{r\to 1}\rho_{p,\infty}(f-f_{r})=0$.
\end{proof}

A density argument similar to the one used in Proposition \ref{desnsity-polynomials} shows that:
\begin{corollary} \label{Cor:density-pol-RM(p,0)} Let $1\leq p <+\infty$.
	Polynomials are dense in $RM(p,0)$.  In particular, $RM(p,0)$ is a separable space.
\end{corollary}

\begin{corollary}
	Let $1\leq p<+\infty$. If $z\in \D$, then \begin{align*}
	\|\delta_{z}\|_{(RM(p,0))^{\ast}}=\|\delta_{z}\|_{(RM(p,\infty))^{\ast}} \quad {\textrm and} \quad 
	\|\delta'_{z}\|_{(RM(p,0))^{\ast}}= \|\delta'_{z}\|_{(RM(p,\infty))^{\ast}}. 
	\end{align*}
	In particular, 
	\begin{align*}
	\|\delta'_{z}\|_{(RM(p,0))^{\ast}}\asymp \frac{\|\delta_{z}\|_{(RM(p,0))^{\ast}}}{1-|z|^2}.
	\end{align*}
\end{corollary}
\begin{proof}
Since $RM(p,0)\subset RM(p,\infty)$,  we have that  $\|\delta_{z}\|_{(RM(p,0))^{\ast}}\leq \|\delta_{z}\|_{(RM(p,\infty))^{\ast}}$ and $\|\delta'_{z}\|_{(RM(p,0))^{\ast}}\leq \|\delta'_{z}\|_{(RM(p,\infty))^{\ast}}$.
	
	Let us see that $\|\delta_{z}\|_{(RM(p,0))^{\ast}}\geq \|\delta_{z}\|_{(RM(p,\infty))^{\ast}}$ and $\|\delta'_{z}\|_{(RM(p,0))^{\ast}}\geq \|\delta'_{z}\|_{(RM(p,\infty))^{\ast}}$.
	If $f\in RM(p,\infty)$ with $\rho_{p,\infty}(f)=1$, by Proposition \ref{propp0unif},  $f_r\in RM(p,0)$ and
	\begin{align*}
	\rho_{p,\infty}(f_r)=\sup_{\theta} \left(\int_{0}^{1} |f(ru e^{i\theta})|^{p}\ du\right)^{1/p}=\sup_{\theta} \left(\int_{0}^{r} |f(u e^{i\theta})|^{p}\ \frac{du}{r}\right)^{1/p}\leq \frac{1}{r}\rho_{p,\infty}(f)
	\end{align*}
	Moreover, it is easy to see that $\delta_{z}(f_r)\rightarrow \delta_{z}(f)$ and $\delta'_{z}(f_r)\rightarrow \delta'_{z}(f)$, when $r\rightarrow 1^{-}$ for a fixed $z\in\D$. 
	
	Fixing $\varepsilon>0$, there exists $f\in RM(p,\infty)$ with $\rho_{p,q}(f)=1$ such that $$|\delta_{z}(f)|\geq \|\delta_{z}\|_{(RM(p,\infty))^{\ast}}-\varepsilon.$$
	In addition, we know that
	\begin{align*}
	|\delta_{z} (f)| &=|f(z)|=\lim\limits_{r\rightarrow 1} |f(rz)|\\
	&=\lim\limits_{r\rightarrow 1} \frac{|\delta_{z}(f_r)|}{\rho_{p,\infty}(f_{r})} {\rho_{p,\infty}(f_{r})}\leq \lim\limits_{r\rightarrow 1}\frac{\|\delta_{z}\|_{(RM(p,0))^{\ast}}}{r}=\|\delta_{z}\|_{(RM(p,0))^{\ast}}.
	\end{align*}
Therefore, it satisfies, for all $\varepsilon>0$,
	\begin{align*}
	\|\delta_{z}\|_{(RM(p,\infty))^{\ast}}\leq \|\delta_{z}\|_{(RM(p,0))^{\ast}}+\varepsilon,
	\end{align*}
	that is, $\|\delta_{z}\|_{(RM(p,\infty))^{\ast}}\leq\|\delta_{z}\|_{(RM(p,0))^{\ast}}$. The proof for $\delta'_{z}$ can be done in a similar way.
\end{proof}

The non-separability of $RM(p,\infty)$ is an easy consequence of the following much deepest result.
\begin{theorem}\label{non-separability}
	Let $1\leq p< \infty$. Then $RM(p,\infty)$ has a subspace isomorphic to $\ell^{\infty}$. Namely, there is a sequence $\{f_k\}$ of functions in $RM(p,0)$
	such that for every $\{\alpha_k\}\in \ell^\infty$ the series $\sum_{k=0}^{\infty}\alpha_{k}f_{k}$ converges uniformly on compact subsets of $\D$ and the operator  
	\begin{align*}
	T:\ell^\infty\rightarrow RM(p,\infty) \quad \mbox{ defined by } \quad T(\{\alpha_{k}\}):=\sum_{k=0}^{\infty}\alpha_{k}f_{k}
	\end{align*}
	establishes an isomorphism between $\ell^\infty$ and $T(\ell^\infty)$.  Moreover, $T(\{\alpha_{k}\})\in RM(p,0)$ if and only if $\{\alpha_k\}\in c_{0}$. In particular, $\sum_{k=0}^\infty f_k\in RM(p,\infty)\setminus RM(p,0)$.
\end{theorem}
\begin{proof}
	For each $k=0,1,2,\dots$, take $r_k=2^{-(k+1)}$, $a_k=1+14^{-(k+1)}$, and 
	\begin{align*}
	\varepsilon_k=\frac{\sqrt[p]{2p-1}}{2^{(k+1)(2-1/p)}} \frac{1}{\sqrt[p]{7^{(k+1)(2p-1)}-1}}.
	\end{align*}
	It is clear that $	\sum_{k=0}^{\infty}r_k=1$,
	\begin{align*}
	\sum_{k=0}^{\infty} \frac{\varepsilon_{k}}{r_{k}^2}&\leq \sqrt[p]{2p-1} \sum_{k=0}^{\infty} \frac{2^{(k+1)/p}}{\sqrt[p]{7^{(k+1)(2p-1)}-1}}\leq \frac{7 \sqrt[p]{2p-1} }{6} \sum_{k=0}^{\infty} \frac{2^{k+1}}{\sqrt[p]{7^{(k+1)(2p-1)}}}\\
	&\leq \frac{7 \sqrt[p]{2p-1} }{6} \sum_{k=0}^{\infty} \left(\frac{2}{7}\right)^{k+1}=\frac{7}{15} \sqrt[p]{2p-1}<1,
	\end{align*}
	and	\begin{align*}
	\int^{1}_{a_k-r_k} \frac{\varepsilon_k^p}{|a_k-r|^{2p}}\ dr=\frac{\varepsilon_{k}^{p}}{{2p-1}}\left(\frac{1}{(a_k-1)^{2p-1}}-\frac{1}{r^{2p-1}_k}\right)=1.
	\end{align*}
	In addition we can find a sequence $\{\theta_k\}$ such that the disks $D(a_{k}e^{\theta_{k}i},r_k)$ are pairwise disjoint. For that, we consider $$\theta_{k}=\arcsin(r_k)+2\sum_{n=0}^{k-1}\arcsin\left(r_n\right).$$
	It is easy to see that $D(a_{k} e^{i\theta_{k}},r_{k})\cap D(a_{k+1} e^{i\theta_{k+1}},r_{k+1})=\emptyset$, because
	\begin{align*}
	\theta_{k+1}-\theta_{k}=\arcsin\left(r_{k+1}\right)+\arcsin\left(r_{k}\right).
	\end{align*}
	Moreover, it is also obtained that 
	\begin{align*}
	|\theta_{k}|\leq \frac{\pi}{2} r_k+\pi \sum_{n=0}^{k-1} r_n< \pi \sum_{n=0}^{k} r_{n} <\pi. 
	\end{align*}
	Finally, take  $f_{k}(z):=\frac{\varepsilon_k}{(ze^{-i\theta_{k}}-a_{k})^2}$, $z\in \C\setminus \{a_k e^{i\theta_k}\}$. Since $f_k$ is bounded in $\D$, it belongs to $RM(p,0)$. 
	In addition, we have that $|f_{k}(z)|\leq \frac{\varepsilon_k}{r_{k}^2}$ if $z\notin D(a_k e^{i\theta_{k}  },r_k)$. 
	
	Since $\sum_{k=0}^{\infty}\frac{\varepsilon_k}{r_{k}^2}<\infty$, it is easy to see that, given a bounded sequence $\{\alpha_k\}$, the sequence  $\{\sum_{n=0}^{k} \alpha_nf_{n}(z)\}$ converges uniformly on compacta of $\D$ to $T(\{\alpha_{k}\}):=\sum_{k=0}^{\infty}\alpha_{k}f_{k}$, so that $T(\{\alpha_{k}\})$ is holomorphic in $\D$. 
	
	By construction, every radius $L_\theta=\{te^{i\theta}\ :\ t\in [0,1)\}$ only cuts one of the open balls.
	Let us see that $f=T(\{\alpha_{k}\})\in RM(p,\infty)$. 
	On the one hand, if   $\theta\in [0,2\pi]$ is such that there is $k_0$ with $e^{i\theta}\in D(a_{k_0}e^{i\theta_{k_0}},r_{k_0})$. Then 
	\begin{align*}
	\left(\int_{0}^{1}|f(re^{i\theta})|^{p}\ dr\right)^{1/p}&  \leq |\alpha_{k_0}|\left(\int_{0}^{1} |f_{k_0}(re^{\theta i})|^{p}\ dr\right)^{1/p}+\sum_{j=0} |\alpha_{k}|\frac{\varepsilon_j}{r_j^2}\\
	&  \leq ||\{\alpha_k\}||_{\ell^\infty} \left( \left(\int_{0}^{1} |f_{k_0}(re^{\theta_{k_0} i})|^{p}\ dr\right)^{1/p}+1\right) \\
	&\leq ||\{\alpha_k\}||_{\ell^\infty} \left(  \left(1+(a_{k_0}-r_{k_0})\frac{\varepsilon_{k_0}^{p}}{r_{k_0}^{2p}}\right)^{1/p}+1\right) \leq 3||\{\alpha_k\}||_{\ell^\infty}.  
	\end{align*}
	On the other hand, if  $\theta\in [0,2\pi]$ is such that $e^{i\theta}\notin  D(a_{k}e^{i\theta_{k}},r_{k})$ for all $k$, then 
	\begin{align*}
	\left(\int_{0}^{1} |f(re^{i\theta})|^{p}\ dr\right)^{1/p}\leq  ||\{\alpha_k\}||_{\ell^\infty}\sum_{k=0}^{\infty}\frac{\varepsilon_{k}}{r_{k}^2}\leq  ||\{\alpha_k\}||_{\ell^\infty}.
	\end{align*}
	That is $f=T(\{\alpha_{k}\})\in RM(p,\infty)$ and, in particular, $T:\ell^\infty\rightarrow RM(p,\infty)$ is bounded.

	Let us see that $T$ is open so that it establishes an isomorphism between $\ell^\infty$ and $T(\ell^\infty)$.  For each $n$, it follows 
	\begin{align*}
	&\rho_{p,\infty}(T(\{\alpha_k\}))\geq \left(\int_{0}^{1}|T(\{\alpha_k\})(re^{i\theta_{n}})|^{p}\ dr\right)^{1/p}\geq \left(\int_{a_n-r_n}^{1}|T(\{\alpha_k\})(re^{i\theta_{n}})|^{p}\ dr\right)^{1/p}\\
	&\geq |\alpha_{n}|-\|\{\alpha_{k}\}\|_{\ell^\infty} \left(\sum_{j=0}^{\infty}\frac{\varepsilon_{j}}{r_{j}^2}\right)(1-a_n+r_n)\geq |\alpha_{n}|-\|\{\alpha_{k}\}\|_{\ell^\infty} \left(\sum_{j=0}^{\infty}\frac{\varepsilon_{j}}{r_{j}^2}\right).
	\end{align*}
	Therefore, $|\alpha_{n}|\leq \|\{\alpha_{k}\}\|_{\ell^\infty} \left(\sum_{j=0}^{\infty}\frac{\varepsilon_{j}}{r_{j}^2}\right)+\rho_{p,\infty}(T(\{\alpha_k\}))$  and taking supremum in $n$ we obtain $$\rho_{p,\infty}(T(\{\alpha_k\})) \geq  \left(1-\sum_{j=0}^{\infty}\frac{\varepsilon_{j}}{r_{j}^2}\right)  \|\{\alpha_{k}\}\|_{\ell^\infty}.$$ Since $\sum_{j=0}^{\infty}\frac{\varepsilon_{j}}{r_{j}^2}<1$, we get that $T$ establishes an isomorphism between $\ell^\infty$ and $T(\ell^\infty)$.

	To end the proof, we show that $T(\{\alpha_{k}\})\in RM(p,0)$ if and only if $\{\alpha_k\}\in c_{0}$.
	
	Let $T\left(\{\alpha_{k}\}\right)\in RM(p,0)$. Then
	\begin{align*}
	& \sup_{\theta}\left(\int_{a_n-r_n}^{1}|T(\{\alpha_k\})(re^{i\theta})|^{p}\ dr\right)^{1/p}\geq \left(\int_{a_n-r_n}^{1}|T(\{\alpha_k\})(re^{i\theta_{n}})|^{p}\ dr\right)^{1/p}\\
	&\geq |\alpha_{n}|-\|\{\alpha_{k}\}\|_{\ell^\infty} \left(\sum_{j=0}^{\infty}\frac{\varepsilon_{j}}{r_{j}^2}\right)(1-a_n+r_n).
	\end{align*}
	Since $1-a_n+r_n\rightarrow 0$ and $T\left(\{\alpha_{k}\}\right)\in RM(p,0)$, it follows that 
	$\{\alpha_k\}\in c_{0}$.

	Conversely,  let $\alpha=\{\alpha_k\}_k\in c_0$ and let us prove that $T(\alpha)\in RM(p,0)$. Since $f_k\in RM(p,0)$, then $\sum_{k=1}^{n} \alpha_k f_{k}\in RM(p,0)$ for all $n\in \N$. Moreover, $\sum_{k=1}^{n} \alpha_k f_{k}\rightarrow T(\alpha)$ because $T$ is continuous and $(\alpha_1,\dots,\alpha_n,0,0,\dots )\rightarrow \alpha$ in $\ell^\infty$. Finally, $T(\alpha)\in RM(p,0)$ since $RM(p,0)$ is a closed subspace of $RM(p,\infty)$.
\end{proof}

\section{Containment relationships}

\subsection{Inclusions}
In this section we will give a characterization for the containment relationships between our spaces. To do this, we recall the notion of the Marcinkiewicz spaces $L^{p,\infty}$, also called the weak $L^{p}$ spaces.

\begin{definition}
	Let $1\leq p<\infty$. We define the weak $L^{p}$ space of measurable functions
	\begin{align*}
	L^{p,\infty}=\left\{f:X\mapsto \C\ \mbox{measurable}: \|f\|_{p,\infty}:=\sup_{t>0}t\lambda_{f}^{1/p}(t)<\infty\right\}
	\end{align*}
	where 
	\begin{align*}
	\lambda_{f}(t)=\mu\left(\left\{x\in X\ :|f(x)|>t\right\}\right).
	\end{align*}
\end{definition}

\begin{lemma}\cite[Proposition 1.1.14, p. 8]{grafakos_classical_fourier_analysis} \label{weakinterpol}
	Let $f\in L^{p_0,\infty}\cap L^{p_1,\infty}$ with $p_0\neq p_1$. Then $f\in L^p$ for $\frac{1}{p}=\frac{1-\lambda}{p_0}+\frac{\lambda}{p_1}$, $\lambda\in (0,1)$. Moreover, there exists a constant $C(p_0,p_1,\lambda)>0$ such that
	\begin{align*}
	\|f\|_{p}\leq C(p_0,p_1,\lambda) \|f\|_{p_0,\infty}^{1-\lambda}\|f\|_{p_1,\infty}^{\lambda},
	\end{align*} 
	for $\lambda\in (0,1)$.
\end{lemma}

\begin{theorem}\label{thm:inclusion}
	Let $1\leq p_0, q_0\leq \infty$ and set 
	\begin{align*}
	A(p_0,q_0)=\left\{ (p,q)\in (0,+\infty]\times(0,+\infty] : \frac{1}{p}+\frac{1}{q}\geq \frac{1}{p_0}+\frac{1}{q_0},\quad p_0\geq p\right\}.
	\end{align*}
	\begin{enumerate}
		\item If $p_0,q_0<+\infty$, then $RM(p_0,q_0)\subset RM(p,q)$ if and only if $(p,q)\in A(p_0,q_0)\setminus \left\{\left(\beta,\infty\right)\right\}$, where $\beta=\frac{p_0q_0}{p_0+q_0}$.
		\item  If either $p_0$ or $q_0$ are $+\infty$, then $RM(p_0,q_0)\subset RM(p,q)$  if and only if $(p,q)\in A(p_0,q_0)$.
	\end{enumerate}
\end{theorem}

Before proving the result, it is worth showing a picture of the set $A(p_0,q_0)$. If $p_0, q_0<+\infty$ then the set $\{(\frac{1}{p},\frac{1}{q}):\ (p,q)\in A(p_0,q_0)\}$ is the grey region (including its boundary) in Figure \ref{p0q0finito} while if either $p_0$ or $q_0$ are $+\infty$, the set $\{(\frac{1}{p},\frac{1}{q}):\ (p,q)\in A(p_0,q_0)\}$ is the grey region  (including its boundary) in Figure \ref{p0infinito} or Figure \ref{q0infinito}, respectively.
	
\begin{center}
\begin{figure}
\begin{subfigure}{.3\textwidth}	
		\begin{tikzpicture}
		\draw[<-] (4,0) -- (3.6,0);
		\draw[->] (3.4,0)-- (0,0) -- (0,3) node[left]{1} -- (0,4);
		\draw[-] (3,0.)node[below]{1}--(3,3)--(0,3);
		\draw [domain=0:3.44] plot (\x, {3.5-\x});
		\draw [-] (2,0)node[below] {$\frac{1}{p_{0}}$}--(2,3);
		\draw [-] (0,1.5) node[left] {$\frac{1}{q_{0}}$}-- (3,1.5);
		\draw[pattern=dots] (0,3)--(0,0)--(3,0)--(3.5,0)--(2,1.5)--(2,3)--(0,3);
		\draw[pattern=dots] (0,3)--(2,3)--(2,4)--(0,4)--(0,3);
		\draw[fill=black, opacity=0.2] (2,3)--(3,3)--(3,0.5)--(3.5,0)--(4,0)--(4,4)--(2,4)--(2,3);
		\draw[fill=black, opacity=0.2] (2,1.5)--(3,1.5)--(3,3)--(2,3)--(2,1.5);
		\draw[fill=black, opacity=0.2]  (3,0.5)--(3,1.5)--(2,1.5)--(3,0.5);
		\draw[line width=0.05cm,black] (2,4)--(2,1.5)--(3.5,0.)--(4,0);
		\draw[black,fill=white] (3.5,0)node[below,black]{$\ \ \ \frac{1}{p_0}+\frac{1}{q_0}$} circle[radius=.1,] ;
		\end{tikzpicture}
		\caption{$p_0, q_0<+\infty$}\label{p0q0finito}
\end{subfigure}%
\begin{subfigure}{.3\textwidth}
		\begin{tikzpicture}
		\draw[<-] (4,0) -- (2.6,0);
		\draw[->] (3,0)-- (0,0)node[below]{0} -- (0,3) node[left]{1} -- (0,4);
		\draw[-] (3,0.)node[below]{1}--(3,3)--(0,3);
		\draw[fill] (2.5,0)node[below]{$\frac{1}{q_{0}}$}  ;
		\draw [domain=0:2.45] plot (\x, {2.5-\x});
		\draw [-] (0,2.5) node[left] {$\frac{1}{q_{0}}$}-- (3,2.5);
		\draw[pattern=dots] (2.5,0)--(0,2.5)--(0,0)--(2.5,0);
		\draw[fill=black, opacity=0.2]  (4,0)--(4,4)--(0,4)--(0,2.5)--(2.5,0)--(4,0);
		\draw[line width=0.05cm,black] (0,4)--(0,2.5)--(2.5,0)--(4,0);
		\end{tikzpicture}
		\caption{$p_0=+\infty$}\label{p0infinito}
		\end{subfigure}%
	\begin{subfigure}{.3\textwidth}	
		\begin{tikzpicture}
		\draw[<-] (4,0) -- (3.6,0);
		\draw[->] (3.4,0)-- (0,0) -- (0,3) node[left]{1} -- (0,4);
		\draw[-] (3,0.)node[below]{1}--(3,3)--(0,3);
		\draw [-] (2,0)node[below] {$\frac{1}{p_{0}}$}--(2,3);
		\draw[pattern=dots] (0,3)--(0,0)--(2,0)--(2,3)--(0,3);
		\draw[pattern=dots] (0,3)--(2,3)--(2,4)--(0,4)--(0,3);
		\draw[fill=black, opacity=0.2] (2,4)--(4,4)--(4,0)--(2,0)--(2,4);
		\draw[line width=0.05cm,black] (2,4)--(2,0)--(4,0);
		\end{tikzpicture}
		\caption{$q_0=+\infty$}\label{q0infinito}
	\end{subfigure}%
\caption{}\label{p0}
\end{figure}
\end{center}

\begin{proof} Bearing in mind that $H^{q_0}=RM(\infty, q_0)$, Proposition \ref{Hardyineq} is nothing but the case $p_0=+\infty$. Therefore, from now on, we will assume that $p_0<+\infty$. To clarify the exposition, we split the proof in several steps.

{\bf Step 1.} If  $p_0,q_0<+\infty$ (Figure \ref{p0q0finito}) and  $(p,q)$ is such that $(1/p,1/q)$ belongs to the open segment with end points $(1/p_0,1/q_0)$ and $(\frac{1}{p_0}+\frac{1}{q_0},0)$, then $RM(p_0,q_0)\subset RM(p,q)$. 	

Write $(\frac{1}{p},\frac{1}{q})=\lambda (\frac{1}{p_0},\frac{1}{q_0})+(1-\lambda) (\frac{1}{p_0}+\frac{1}{q_0},0)$ for some $\lambda\in (0,1)$. Take  $f\in RM(p_{0},q_{0})$ with $\rho_{p_0,q_0}(f)\leq 1$. 
For each $\theta$, define $f_\theta(r):=f(re^{i\theta})$. Let us see that $f_\theta\in L^{p_0,\infty}([0,1])\cap L^{\alpha,\infty}([0,1])$ for almost every $\theta\in [0,2\pi]$, where $\frac{1}{p_0}+\frac{1}{q_{0}}=\frac{1}{\alpha}$. Since $f\in RM(p_0,q_0)$, by the very definition, we have that $f_\theta\in L^{p_0}([0,1))$ for almost every $\theta$ and $\|f_\theta\|_{p_0,\infty}\leq \|f_\theta\|_{p_0}$. Moreover, by Proposition~\ref{estimevfunc}, there is a constant $C>0$ such that 
$|f(z)|\leq C\frac{1}{(1-|z|^2)^{\frac{1}{p_0}+\frac{1}{q_0}}}$, for all $z$, and thus
\begin{align*}
	\|f_\theta\|_{\alpha,\infty}=\sup_{t\geq 0} t \mu(\left\{|f_\theta|>t \right\})^{1/\alpha}\leq \sup_{t\geq 0} t \mu\left(\left\{1-r\leq \frac{C^\alpha}{t^\alpha} \right\}\right)^{1/\alpha}=\sup_{t\geq 0} \min\{t,C\}\leq C,
\end{align*}
so that $f_\theta \in  L^{\alpha,\infty}([0,1])$ for all $\theta$. 
	Hence, applying Lemma~\ref{weakinterpol} we have 
	\begin{align*}
	\|f_\theta\|_{p}\leq C(p_0,\alpha,\lambda)\cdot \|f_\theta\|_{p_{0},\infty}^{\lambda}\cdot \|f_\theta\|^{1-\lambda}_{\alpha,\infty}.
	\end{align*}
 Thus	\begin{align*}
	\begin{split}
	\rho_{p,q}(f)&\leq C(p_0,\alpha,\lambda) \left(\int_{0}^{2\pi} \|f_\theta\|_{p_0,\infty}^{\lambda q}\cdot \|f_\theta\|_{\alpha,\infty}^{(1-\lambda) q}\ d\theta \right)^{1/q}\\
	&\leq C(p_0,\alpha,\lambda)\ C^{1-\lambda} \left(\int_{0}^{2\pi} \|f_\theta\|_{p_{0}}^{\lambda q}\ d\theta \right)^{1/q}
	=C(p_0,\alpha,\lambda)\ C^{1-\lambda} \left(\int_{0}^{2\pi} \|f_\theta\|_{p_{0}}^{q_0}\ d\theta \right)^{1/q}\\
	&=C(p_0,\alpha,\lambda)\ C^{1-\lambda}\rho_{p_0,q_0}(f)^{\lambda}\leq C(p_0,\alpha,\lambda)\ C^{1-\lambda}.
	\end{split}
	\end{align*}
	
\noindent {\bf Step 2.} If $\frac{1}{p}>\frac{1}{p_0}+\frac{1}{q_0}$, then $RM(p_0,q_0)\subset RM(p,\infty)$. 

Take $f\in RM(p_0,q_0)$. By Proposition~\ref{estimevfunc}, there is $C>0$ such that
\begin{equation*}
\begin{split}
\rho_{p,\infty} (f)&=\esssup_ {\theta\in[0,2\pi)}\left(\int_{0}^{1} |f(r e^{i \theta})|^p \ dr \right)^{1/p}\\
&\leq C\esssup_ {\theta\in[0,2\pi)}\left(\int_{0}^{1} \frac{\rho_{p_0,q_0} (f)^p}{(1-r)^{(\frac{1}{p_0}+\frac{1}{q_0})p}}  
\ dr \right)^{1/p}<+\infty
\end{split}
\end{equation*}

\noindent  {\bf Step 3.}  If $p_0\geq  p$ and $q_0\geq  q$ then $RM(p_0,q_0)\subset RM(p,q)$. 

This inclusion is a direct consequence of 	H\"oder's inequality.

Denote by $B(p_0,q_0)=\left\{ (p,q)\in \R^{+}\times\R^{+} : p_0\geq p, \ q_0\geq q\right\}$ and $(p_\lambda,q_\lambda)$ the couple such that  $(\frac{1}{p_\lambda},\frac{1}{q_\lambda})=\lambda (\frac{1}{p_0},\frac{1}{q_0})+(1-\lambda) (\frac{1}{p_0}+\frac{1}{q_0},0)$.
Since 
$$
A(p_0,q_0)\setminus \left\{\left(\frac{p_0q_0}{p_0+q_0},\infty\right)\right\}=\cup_{\lambda\in (0,1]} B(p_\lambda,q_\lambda)\cup \left\{(p,\infty):\, \frac{1}{p}>\frac{1}{p_0}+\frac{1}{q_0}\right\},
$$
Steps 1, 2 and 3 give that if $p_0, q_0<+\infty$ and $(p,q)\in A(p_0,q_0)\setminus \left\{\left(\frac{p_0q_0}{p_0+q_0},\infty\right)\right\}$ then $RM(p_0,q_0)\subset RM(p,q)$.

\noindent  {\bf Step 4.} If $RM(p_0,q_0)\subset RM(p,q)$ then $p_0\geq p$.

	By closed graph theorem there is a constant $C>0$ such that $\rho_{p,q}(f)\leq C\rho_{p_0,q_0}(f)$ for all $f\in RM(p_0,q_0)$. Taking $f_{n}(z)=z^n$ we obtain 
	\begin{align*}
	\rho_{p,q}(f_n)=(1+np)^{-\frac{1}{p}}\leq C (1+np_0)^{-\frac{1}{p_0}}=\rho_{p_0,q_0}(f_n)
	\end{align*}	
and this inequality holds for all $n$ if and only if $p_0\geq p$.

\noindent  {\bf Step 5.} If $\frac{1}{p_0}+\frac{1}{q_0}>\frac{1}{p_1}+\frac{1}{q_1}$ then $RM(p_0,q_0)\nsubseteq RM(p_1,q_1)$

	We consider a function $f_\alpha$ of Example~\ref{ex1} such that $\frac{1}{p_1}+\frac{1}{q_1}<\alpha<\frac{1}{p_0}+\frac{1}{q_0}$. Hence, we have a function $f_\alpha$ such that $f_\alpha\in RM(p_0,q_0)\setminus RM(p_1,q_1)$.
	
\noindent  {\bf Step 6.} If $p_0, q_0<+\infty$, then $RM(p_0,q_0)\nsubseteq RM(\beta,\infty)$, where $\beta=\frac{p_0q_0}{p_0+q_0}$.

Assume that $RM(p_0,q_0)\subset RM(\beta,\infty)$.
	By closed graph theorem there is a positive constant $C>0$ such that $\rho_{\beta,\infty}(f)\leq C\rho_{p_0,q_0}(f)$. For each $n$, consider the function $f_{n,\beta}$ introduced in Example~\ref{ex3}. Then 
	\begin{align*}
	\rho_{\beta,\infty}(f_{n,\beta})\geq \left(\int_{0}^{1}\sum_{k=0}^{n} r^k\ dr\right)^{\frac{1}{\beta}}=\left(\sum_{k=0}^{n} \frac{1}{k+1}\right)^{\frac{1}{\beta}}\geq \ln^{\frac{1}{\beta}}(n+1).
	\end{align*}
	Thus, Example~\ref{ex3} would imply 
	\begin{align*}
	\ln^{1/\beta}(n+1)\leq C \left(\frac{p_0}{p_0-\beta}\right)^{1/p_0}\ln^{1/q_0}(n+1),
	\end{align*}
	what is not possible if $n$ is large enough. So  $RM(p_0,q_0)\nsubseteq RM(\beta,\infty)$.

Clearly, Steps 4, 5 and 6 imply that if $RM(p_0,q_0)\subset RM(p,q)$ then $(p,q)\in A(p_0,q_0)\setminus \left\{\left(\beta,\infty\right)\right\}$. Therefore, statement (1) and (2) are proved. 
\end{proof}

A simple argument shows that if $q<+\infty$, the density of the polynomials in $RM(p,q)$ implies that if $RM(p,q)\subset RM(p_0,\infty)$ if and only if $RM(p,q)\subset RM(p_0,0)$.

The situation is not so clear to study when $RM(p_0,0)$ is contained in $RM(p,q)$. To characterize it, we need the following lemma.

\begin{lemma}\label{converge-uniform}
Let $1\leq p,q\leq +\infty$. If $\{f_{n}\}$ is a bounded sequence in $RM(p,q)$ that converges uniformly on compact subsets of the unit disc to $f$. Then $f\in RM(p,q)$. 
\end{lemma}
\begin{proof} Clearly the function $f$ is holomorphic. Assume that $p,q<+\infty$. By Fatou's Lemma, for each $\theta$ we have
$$
\int_{0}^{1}|f(re^{i\theta})|^{p}\, dr\leq g(\theta):=\liminf_{n} g_{n}(\theta),
$$
where, for each $n$, $g_{n}(\theta) :=\int_{0}^{1}|f_{n}(re^{i\theta})|^{p}\, dr$.
Repeating again the argument, we have
\begin{equation*}
\begin{split}
\rho_{p,q}(f)^{q}&\leq 
\frac{1}{2\pi}\int_{0}^{2\pi} \left(g(\theta)\right)^{q/p}\, d\theta\leq 
\liminf_{n}\frac{1}{2\pi}\int_{0}^{2\pi} \left(g_{n}(\theta)\right)^{q/p}\, d\theta \\
&=\liminf_{n} \rho_{p,q}^q(f_{n})\leq \sup_{n} \rho_{p,q}^q(f_{n})<+\infty.
\end{split}
\end{equation*}
A similar argument works in the remaining cases, so that we are done.
\end{proof}

\begin{proposition}
Let $1\leq p_0\leq +\infty$. Then $RM(p_0,\infty)\subset RM(p,q) $ if and only if $RM(p_0,0)\subset RM(p,q) $. 
\end{proposition}
\begin{proof} Assume that $RM(p_0,0)\subset RM(p,q) $. Take $f\in RM(p_0,\infty)$. For each $r<1$, the function $f_{r}$ belongs to $RM(p_0,0)$ and then to $RM(p,q) $. Since $\{f_{r}:\, r<1\}$ is bounded in $RM(p_0,0)$, it is also bounded in $RM(p,q) $. Since $f_{r}$ converges uniformly on compact subset of $\D$ to $f$, Lemma \ref{converge-uniform} guarantees that $f\in RM(p,q)$.
\end{proof}

\subsection{Compactness of the inclusions}
Once the containment relationships of these spaces have been determined, we study  when such inclusions are compact.

A standard argument shows the following characterization of compactness.

\begin{lemma}\label{lemmacompactness}
Let $1\leq p_0,q_0\leq +\infty$ and $1\leq p,q\leq +\infty$. Then $i:RM(p_0,q_0)\rightarrow RM(p,q)$ is compact if and only if every bounded sequence $\{f_n\}$ in $RM(p_0,q_0)$ that converges to zero uniformly on compact subsets of the unit disc satisfies that $ \lim_n\rho_{p,q} (f_n)=0$.
\end{lemma} 

We will use this lemma several times in the proof of the next theorem without explicit reference. We also need the following result.

	\begin{proposition}\label{notangpinf}
	Let $1\leq p<+\infty$, $f\in RM(p,\infty)$, and $\sigma\in \partial \D$. Then for the non-tangential limit we have $\angle\lim\limits_{z\rightarrow \sigma}  f(z)(1-\overline \sigma z)^{1/p}=0$.
\end{proposition}
\begin{proof} Without loss of generality we assume that $\sigma=1$. Suppose that $\rho_{p,\infty}(f)\leq 1$ and   consider the holomorphic function $h(z)=f(z)(1-z)^{1/p}$. Fix $R>1$ and the Stolz region $S(1,R)=\{z\in \D: |1-z|<R(1-|z|)\}$. Looking at \eqref{eqfung} in the proof of Proposition \ref{estimevfunc}, we see that there is a constant $C$ such that 
	$$|h(z)|\leq {R^{1/p}}|f(z)|(1-|z|)^{1/p}\leq C R^{1/p}, \quad z\in S(1,R).$$
	That is, the function $h$ is bounded on $S(1,R).$ Therefore, by Lindel\"of's Theorem \cite[Theorem 1.5.7, p. 26]{Bracci_Contreras_Diaz-Madrigal}, it is enough to prove that 
	$\lim_{r\rightarrow 1^{-}} |f(r)|(1-r)^{1/p}=0$. 
	
	Assume by contradiction that there is a constant $c_1>0$ and a sequence $\{r_k\}$ where $r_k\rightarrow 1^{-}$ such that $c_1\leq |f(r_k)|(1-|r_k|)^{1/p}$ for all $k$.
	Write  $\delta_{k}:=1-r_{k}$. By Proposition \ref{estiderfunc}, there is a constant $C$ such that 
	$|f'(x)|\leq\frac{C}{(1-x)^{1+\frac{1}{p}}}$ for all $x\in (0,1)$. Choose $ \varepsilon<\frac{c_1}{2C}$. Then,  for  $1-(1+\varepsilon)\delta_{k}<x<1-\delta_{k}$, 
	\begin{align*}
	|f(x)-f(r_k)|\leq C |x-r_k| \frac{1}{(1-r_k)^{1+1/p}}\leq C \varepsilon \delta_{k}\frac{1}{\delta_{k}^{1+1/p}}=\frac{C\varepsilon}{\delta_{k}^{1/p}}<\frac{c_1}{2\delta_{k}^{1/p}}.
	\end{align*}
	Thus
	$$
	|f(x)|\geq |f(r_k)|-|f(x)-f(r_k)|\geq  \frac{c_1}{(1-r_k)^{1/p}}- \frac{c_1}{2\delta_{k}^{1/p}}= \frac{c_1}{2\delta_{k}^{1/p}}
	$$
	and 
	\begin{align*}
	\left(\int_{1-(1+\varepsilon)\delta_{k}}^{1-\delta_{k}} |f(x)|^{p}\ dx\right)^{1/p}\geq (\varepsilon \delta_{k})^{1/p} \frac{C_1}{2\delta_{k}^{1/p}}=\frac{C_1 \varepsilon^{1/p}}{2}.
	\end{align*}
	But, this is impossible because $\int_{0}^{1} |f(x)|^{p}\ dx<+\infty$.
\end{proof}

\begin{theorem}\label{thm:compactness}
	Let $1\leq p_0,q_0\leq +\infty$ and $1\leq p,q\leq +\infty$. Then $i:RM(p_0,q_0)\rightarrow RM(p,q)$ is compact if and only if $\frac{1}{p}+\frac{1}{q}>\frac{1}{p_0}+\frac{1}{q_0}$ and $p<p_0$.
\end{theorem}

As we can see in the Figure \ref{p0compacto}, the grey region, removing this time the dotted lines, represents the spaces $RM(p,q)$ such as $i:RM(p_0,q_0)\rightarrow RM(p,q)$ is compact when $p_0<+\infty$ in Figure \ref{p0finitocompacto} and when $p_0=+\infty$ in Figure \ref{p0infinitocompacto}.

\begin{center}
\begin{figure}
\begin{subfigure}{.4\textwidth}
	\begin{tikzpicture}[scale=1, every node/.style={scale=1}]
	\draw[<-] (4.5,0) -- (3.6,0);
	\draw[->] (3.4,0)-- (0,0) -- (0,3) node[left]{1} -- (0,4);
	\draw[-] (3,0.)node[below]{1}--(3,0.5);
	\draw[-] (2,3)--(0,3);
	\draw[-](3,0.5)--(3,3)--(2,3);
	\draw [domain=0:3.44] plot (\x, {3.5-\x});
	\draw [-] (2,0)node[below] {$\frac{1}{p_{0}}$}--(2,3);
	\draw [-] (0,1.5) node[left] {$\frac{1}{q_{0}}$}-- (3,1.5);
	\draw[fill=black, opacity=0.2] (2,1.5)--(3,1.5)--(3,3)--(2,3)--(2,1.5);
	\draw[fill=black, opacity=0.2]  (2,1.5)--(3.,0.5)--(3,1.5)--(2,1.5);
	\draw[line width=3.5,black,dashed] (2,4)--(2,1.5)--(3,0.5);
	\draw[black, fill=black] (3,0.5) circle[radius=0.05] ; 
	\draw[fill,white] (3.5,0)node[below,black]{$\ \ \ \frac{1}{p_0}+\frac{1}{q_0}$} circle[radius=.09] ;
	\draw[black] (3.5,0) circle[radius=0.09] ;
	\end{tikzpicture}
	\caption{$p_0<+\infty$}\label{p0finitocompacto}
	\end{subfigure}%
\begin{subfigure}{.4\textwidth}
	\begin{tikzpicture}[scale=1, every node/.style={scale=1}]
	\draw[<-] (4.5,0) -- (2.5,0);
	\draw[-] (2.5,0)-- (0,0)node[below]{0} -- (0,3) node[left]{1};
	\draw[line width=3.5 ,black] (0,3)  -- (0,4);
	\draw[-, black] (3,0.)node[below,black]{1}--(3,3)--(0,3);
	\draw [domain=0:2.45] plot (\x, {2.5-\x});
	\draw [-] (0,2.5) node[left] {$\frac{1}{q_{0}}$};
	\draw[fill=black, opacity=0.2] (0,2.5)--(3,2.5)--(3,3)--(0,3)--(0,2.5);
	\draw[fill=black, opacity=0.2]  (2.5,0)--(3,0)--(3,2.5)--(0,2.5)--(2.5,0);
	\draw[line width=3.5,black,dashed] (0,3)--(0,2.5)--(2.5,0);
	\draw[-] (0,2.5) -- (3,2.5);
	\draw [fill=white] (2.4,-0.1) rectangle (2.6,0.1);
	\draw[fill,black] (2.5,0)node[below]{$\frac{1}{q_{0}}$};
	\end{tikzpicture}
	\caption{$p_0=+\infty$}\label{p0infinitocompacto}
\end{subfigure}%
\caption{}\label{p0compacto}
\end{figure}
\end{center}

\begin{proof}
Bearing in mind Theorem \ref{thm:inclusion}, we have to prove that the inclusion is compact if $\frac{1}{p}+\frac{1}{q}>\frac{1}{p_0}+\frac{1}{q_0}$ and $p<p_0$ and it is not compact if either $\frac{1}{p}+\frac{1}{q}=\frac{1}{p_0}+\frac{1}{q_0}$ or  $p=p_0$.

Let us start by showing that it is not compact if $p=p_0$.  For each $n$, consider the function $f_{n}(z)=(np_0+1)^{1/p_0}z^n$, $z\in \D$. A simple calculation shows that $\rho_{p_0,q_0}(f_n)=1$ and that the sequence $\{f_n\}$ converges uniformly to zero on compacts of the unit disc. Assume that $i: RM(p_0,q_0)\rightarrow RM(p_0,q)$ is compact, then there exists a subsequence $\{f_{n_k}\} $ such that $\rho_{p_0,q}(f_{n_k})$ must go to $0$ as $k$ goes to $\infty$. But this is not possible because $\rho_{p_0,q}(f_n)=1$ for all $n$. 

Take now $p$ and $q$ such that  $\frac{1}{p}+\frac{1}{q}=\frac{1}{p_0}+\frac{1}{q_0}$. Assume that $i:RM(p_0,q_0)\rightarrow RM(p,q)$ is compact. Then  $i^{\ast}:(RM(p,q))^{\ast}\rightarrow (RM(p_0,q_0))^{\ast}$ is also a compact operator. We assume that $q<+\infty$.
	
Let us see that $\frac{\delta_z}{\|\delta_z\|_{\left(RM(p,q)\right)^{\ast}}}$ $w^{\ast}$-converges to $0$ when $|z|\rightarrow 1$. Taking $p$ a polynomial we obtain 
	\begin{align*}
\frac{|\delta_z(p)|}{\|\delta_z\|_{\left(RM(p,q)\right)^{\ast}}}\asymp |p(z)|(1-|z|^2)^{\frac{1}{p}+\frac{1}{q}}\leq \|p\|_{\infty}(1-|z|^2)^{\frac{1}{p}+\frac{1}{q}},
	\end{align*}
which clearly goes to $0$ 	when $|z|\rightarrow 1$. 
Since $q<+\infty$, by Proposition \ref{desnsity-polynomials},  polynomials  are dense in  $RM(p,q)$ and then $\frac{\delta_z}{\|\delta_z\|_{\left(RM(p,q)\right)^{\ast}}}$ $w^{\ast}$-converges to $0$ when $|z|\rightarrow 1$.
 Therefore, the compactness of $i^*$ gives
	\begin{align*}
	\lim_{|z|\rightarrow 1}\left\|i^{\ast}\left(\frac{\delta_z}{\|\delta_z\|_{\left(RM(p,q)\right)^{\ast}}}\right)  \right\|_{(RM(p_0,q_0))^{\ast}}=0 .
	\end{align*}
However, this is impossible because, as we shall now see, such norm must be greater than a certain positive constant.
	Indeed, let $f\in RM(p_0,q_0)$. We have 
	\begin{align*}
	\left|\left<f,i^{\ast}\left(\frac{\delta_z}{\|\delta_z\|_{\left(RM(p,q)\right)^{\ast}}}\right)\right>\right|=\left|\left<f,\frac{\delta_z}{\|\delta_z\|_{\left(RM(p,q)\right)^{\ast}}}\right>\right|=\frac{|f(z)|}{\|\delta_z\|_{\left(RM(p,q)\right)^{\ast}}}
	\end{align*}
	and, by  Proposition \ref{estimevfunc},
	\begin{align*}
\left\|i^{\ast}\left(\frac{\delta_z}{\|\delta_z\|_{\left(RM(p,q)\right)^{\ast}}}\right)  \right\|_{(RM(p_0,q_0))^{\ast}}\geq \frac{\|\delta_z\|_{\left(RM(p_0,q_0)\right)^{\ast}}}{\|\delta_z\|_{\left(RM(p,q)\right)^{\ast}}}\asymp 1.
	\end{align*}
The argument for the case $p_0=q=\infty$ is the same. However, we consider a sequence $\{z_{n}\}$ in the Stolz region such that $|z_n|\rightarrow 1$. In this way, we obtain the $w^{\ast}$-convergence bearing in mind  Proposition~\ref{notangpinf}.

Assume now that $\frac{1}{p_1}>\frac{1}{p_0}+\frac{1}{q_0}$ and take a sequence $\{f_n\}$ in $RM(p_0,q_0)$ such that $\rho_{p_0,q_0}(f_n)\leq 1$ for all $n$ and it converges to zero uniformly on compact subsets of $\D$. We claim that $\lim_n\rho_{p_1,\infty}(f_n)=0$. Otherwise, there is $\epsilon >0$ and a subsequence (that we denote equal) such that $\rho_{p_1,\infty}(f_n)>\epsilon$ for all $n$. Thus, we find $\{\theta_n\}$ such that 
\begin{equation}\label{Eq:compactness}
\int_0^1|f_n(re^{i\theta_n})|^{p_1} dr\geq \epsilon^{p_1},
\end{equation}
for all $n\in \N$. For each $n$, we write $g_n(r):=f_n(re^{i\theta_n})$, $r\in [0,1)$. Since $\rho_{p_0,q_0}(f_n)\leq 1$, there is a constant $C>0$ such that 
$$
|g_n(r)|=|f_n(re^{i\theta_n})|\leq \frac{C}{(1-r^2)^{\frac{1}{p_0}+\frac{1}{q_0}}}.
$$
Since the map $r\mapsto \frac{C}{(1-r^2)^{\frac{1}{p_0}+\frac{1}{q_0}}}$ belongs to $L^{p_1}([0,1])$ and $\{g_n\}$ converges pointwise to zero, we get that it converges to zero in the norm of $L^{p_1}([0,1])$ which contradicts \eqref{Eq:compactness}. So that the claim holds. 

Take now $p,q$ such that there is $\lambda\in (0,1)$ with $\left(\frac{1}{p},\frac{1}{q}\right)= \lambda \left(\frac{1}{p_0},\frac{1}{q_0}\right)+(1-\lambda)\left(\frac{1}{p_1},0\right)$. Then, for each $f\in RM(p,q)$,
\begin{equation*}
\begin{split}
\int_0^1|f(re^{i\theta})|^p\, dr&=\int_0^1|f(re^{i\theta})|^{\lambda p}|f(re^{i\theta})|^{(1-\lambda)p}\, dr \\
&\leq \left(\int_0^1|f(re^{i\theta})|^{ p_0}\, dr\right)^{\lambda p/p_0}\left(\int_0^1|f(re^{i\theta})|^{ p_1}\, dr\right)^{(1-\lambda) p/p_1}.
\end{split}
\end{equation*}
So that
\begin{equation*}
\rho_{p,q}(f)\leq \rho_{p_0,q_0}(f)^\lambda \rho_{p_1,\infty}(f)^{1-\lambda}. 
\end{equation*}
This inequality, the above claim and Lemma \ref{lemmacompactness} show that $i\colon RM(p_0,q_0)\rightarrow RM(p,q)$ is compact whenever $\frac{1}{p}+\frac{1}{q}>\frac{1}{p_0}+\frac{1}{q_0}$ and $q_0<q$. 

Take now $p,q$ such that $\frac{1}{p}+\frac{1}{q}>\frac{1}{p_0}+\frac{1}{q_0}$,  $p<p_0$ and $q_0\geq q$. Fix $\tilde q<q$ such that $\frac{1}{p_0}+\frac{1}{\tilde{q}}<\frac{1}{p}+\frac{1}{q}$. By the above argument, the inclusion map $\tilde i$ from  $RM(p_0,\tilde q)$  into $RM(p, q)$ is compact. Since $i:RM(p_0,q_0)\to RM(p,q)$ factorizes through  $\tilde i$, we get that $i$ is compact.
\end{proof}

\section{Bergman projection}

In the theory of Banach spaces of analytic functions, a useful integral operator is the Bergman projection  
\begin{align*}
P(f)(z)=\int_{\D} K(z,w) f(w)\ dA(w),\quad z\in \D,
\end{align*}
with kernel 
\begin{align}\label{Bergmankernel}
K(z,w)=(1-z\overline{w})^{-2},\quad z,w\in\D,
\end{align} 
which is called the Bergman kernel. Such function is the reproducing kernel for the Bergman space $A^2$.

This projection is well-defined on $L^1(\D)$, mapping each function of $L^1(\D)$ to an analytic function and mapping each function of the Bergman space $A^1$ into itself.
Moreover, for $1<p<\infty$ it is known that the Bergman projection is a bounded operator from $L^p(\D)$ onto $A^p$. 
This property allows to describe the dual of Bergman spaces $A^p$.
\begin{theorem}
For $1<p<+\infty$, the dual space of $A^p$ can be identified with $A^{p'}$, where $p'$ is the conjugated index, that is, $\frac{1}{p}+\frac{1}{p'}=1$. Such functional $\phi\in (A^p)^\ast$ has a unique representation
\begin{align*}
\phi(f)=\phi_{g}(f)=\int_{\D} f(w)\overline{g(w)}\ dA(w),\quad f\in A^p,
\end{align*}
for some $g\in A^{p'}$.
\end{theorem}

Mimicking this schedule for Bergman spaces (but with a  deeper argument), in this section we prove the boundedness of the Bergman projection from $L^q(\mathbb{T},L^p[0,1])$ onto $RM(p,q)$, where $1<p,q<\infty$ and, as a byproduct, we describe its dual.

To study the duality of $RM(p,q)$ spaces, the following theorem will be important because it provides a characterization of the dual space of $L^q(\mathbb{T},L^p[0,1])$, for $1\leq p,q<+\infty$.

\begin{theorem}{\cite[Theorem 1, p. 304]{benedek_panzone_1962}}\label{panzone_thm}
	Let $1\leq p,q<+\infty$. $J(f)$ is a continuous functional on the normed space $L^q(\mathbb{T},L^p[0,1])$ if and only if it can be represented by
	\begin{align*}
	J(f)=\int_{\D} h(w)f(w)\ dA(w)
	\end{align*}
	where $h(w)$ is a uniquely determined function of $L^{q'}(\mathbb{T},L^{p'}[0,1])$ and $\|J\|=\rho_{p',q'}(h)$.
\end{theorem}


\begin{theorem}\label{BergmanProj}
	Let $1<p,q<+\infty$. The Bergman projection $P$ is bounded from  the space $L^q(\mathbb{T},L^p[0,1])$ onto $RM(p,q)$.
\end{theorem}

	Since the restriction of $P$ to $RM(p,q)$ is the identity and $Pf$ is analytic for all $f\in L^{q}(\mathbb{T},L^p[0,1])$, in order to prove above theorem it is enough to show that  $P$ is bounded from $ L^q(\mathbb{T},L^{p}[0,1])$ into itself. 

Before going into the proof of this result, we introduce some necessary terminology. 
In general, given a measurable function $M:\D\times\D\mapsto \C$ we can define the integral operator
\begin{align*}
T_{M}(f)(re^{i\theta})&=\int_{\D} M(re^{i\theta},w) f(w)\ dA(w)\\
&=\int_{0}^{2\pi} \hspace{-0.5em}\int_{0}^{1} M(re^{i\theta},\rho e^{i\varphi})\ f(\rho e^{i\varphi})\ \rho \frac{d\rho d\varphi}{\pi},\quad r\in[0,1),\ \theta\in[0,2\pi],
\end{align*}
whernever such integral exists. 

	From now on,  with a little abuse of notation, $|\theta-\varphi|$ will denote the distance between $\theta$ and $\phi$ in the quotient group $\R/2\pi \Z$, that is, $\min_{k\in \Z} |\theta-\varphi+2k\pi|.$
	Notice also that in order to prove the boundedness of $P=T_{K}$,  it is sufficient to check the boundedness of $T_{\tilde K}$, where $$\tilde{K}(re^{i\theta},\rho e^{i\varphi}):=K(re^{i\theta},\rho e^{i\varphi}) \chi_{\{|\theta-\varphi|\leq 1\}},$$ because $K-\tilde{K}$ is a bounded function.

	Moreover, by showing that $T_D: L^q(\mathbb{T},L^{p}[0,1])\rightarrow L^q(\mathbb{T},L^{p}[0,1])$ is bounded, where
	\begin{align*}
	D(re^{i\theta},\rho e^{i\varphi})=
	\begin{cases}
	0, & \mbox{if} \quad |\theta-\varphi|\geq 1,\\
	\frac{1}{|\varphi-\theta|^2}, & \mbox{if} \quad 1\geq |\theta-\varphi|\geq 1-r\rho,\\
	\frac{1}{(1-r\rho)^2}, & \mbox{if} \quad |\theta-\varphi|\leq 1-r\rho\\
	\end{cases}
	\end{align*} we obtain the boundedness of $P: L^q(\mathbb{T},L^{p}[0,1])\rightarrow L^q(\mathbb{T},L^{p}[0,1])$ since $$|K(re^{i\theta},\rho e^{i\varphi})| \chi_{\{|\theta-\varphi|\leq 1\}}\leq 4 D(re^{i\theta},\rho e^{i\varphi}).$$
	
	Bearing in mind the change of variable $x=1-r$ and $y=1-\rho$, it follows that $\frac{\tilde{H}(\theta,\varphi,x,y)}{4}\leq D(\theta,\varphi,1-x,1-y)\leq \tilde{H}(\theta,\varphi,x,y)$, $x,y\in[0,1]$, with
	\begin{align*}
	\tilde{H}(\theta,\varphi,x,y)=
	\begin{cases}
	0, & \mbox{if} \quad |\theta-\varphi|\geq 1,\\
	\frac{1}{|\theta-\varphi|^2}, & \mbox{if} \quad 1\geq |\theta-\varphi|\geq \max\{x,y\},\\
	\frac{1}{(\max\{x,y\})^2}, & \mbox{if} \quad \max\{x,y\}\geq |\theta-\varphi|,\\
	\end{cases}
	\end{align*}
	because $\max\{x,y\}\leq 1-r\rho\leq 2\max\{x,y\}$. 

Finally, next lemma shows that	
 the boundedness of the operator $T_{\tilde{H}}:  L^q(\mathbb{T},L^{p}[0,1])\rightarrow  L^q(\mathbb{T},L^{p}[0,1])$ is equivalent to the boundedness of $T_{H}:  L^q(\mathbb{T},L^{p}[0,1])\rightarrow  L^q(\mathbb{T},L^{p}[0,1])$, where
	\begin{align*}
	{H}(\theta,\varphi,x,y)=
	\begin{cases}
		0, & \mbox{if}\quad |\varphi-\theta|>1 \quad \mbox{or} \quad \max\{x,y\}>|\varphi-\theta|,\\
		\frac{1}{|\varphi-\theta|^2}, & \mbox{if}\quad 1\geq |\varphi-\theta|\geq \max\{x,y\}.
	\end{cases}
	\end{align*}
	
	\begin{remark}\label{remarkdilat}
		Let $a,b\in (0,1]$. If we have the following relation $J(\theta, \varphi, x,y)=K(\theta,\varphi,ax,by)$ between the kernels $J$ and $K$, then 
		$$\|T_J\colon L^q(\mathbb{T},L^{p}[0,1])\rightarrow L^q(\mathbb{T},L^{p}[0,1])\|\leq \frac{b^{1/p}}{ba^{1/p}} \|T_K\colon L^q(\mathbb{T},L^{p}[0,1])\rightarrow L^q(\mathbb{T},L^{p}[0,1])\|.$$
	\end{remark}
	\begin{lemma}
		The operator $T_H:  L^q(\mathbb{T},L^{p}[0,1])\rightarrow  L^q(\mathbb{T},L^{p}[0,1])$ is bounded if and only if the operator $T_{\tilde{H}}:  L^q(\mathbb{T},L^{p}[0,1])\rightarrow  L^q(\mathbb{T},L^{p}[0,1])$ is bounded.
	\end{lemma}

	\begin{proof}
		Clearly, the boundedness of $T_{\tilde{H}}$ implies the boundedness of $T_{H}$ because  $0\leq H\leq \tilde{H}$. Now, we proceed to show the converse implication. First of all, we define the dilated kernels  $H_n(\theta,\varphi,x,y):=2^{-2n} H(\theta,\varphi,2^{-n}x,2^{-n}y)$. Using Remark~\ref{remarkdilat} and denoting by $\Vert \cdot\Vert$ the operator norm from $L^{q}(\mathbb{T},L^p[0,1]$ into itself, we have 
		
	\begin{align}\label{ineqdltn}
\|T_{H_n}\|\leq 2^{-n}\|T_{H}\|.
	\end{align}
		Therefore, using the fact that
		$$\tilde{H}(\theta,\varphi,x,y)\leq 3 \sum_{n=0}^{\infty} H_n(\theta,\varphi,x,y)$$  
		and the previous inequality \eqref{ineqdltn}, we conclude 
		\begin{align*}
		\|T_{\tilde{H}} \|\leq 3\sum_{n=0}^{\infty} \|T_{H_n}\|\leq 3\sum_{n=0}^{\infty} 2^{-n} \|T_{H}\| \leq 6\|T_{H}\|,
		\end{align*}
		and we are done.
		\end{proof}

	\begin{lemma}\label{PRineq}
		Let $f\in L^{q}(\mathbb{T},L^p[0,1])$, $g\in L^{q'}(\mathbb{T},L^{p'}[0,1])$ such that $f,g\geq 0$. Then 
		\begin{align*}
		\int_{0}^{2\pi}\hspace{-0.5em} \int_{0}^{1} (T_{H}f)\ g\ dx\ d\theta\leq \int_{0}^{2\pi} \hspace{-0.5em}\int_{0}^{2\pi} Rf(\theta, |\varphi-\theta|) Rg(\varphi, |\varphi-\theta|)\ d\theta\ d\varphi,
		\end{align*}
		where $Rf(\theta,x)=\begin{cases}
	\sup_{1\geq t\geq x} \frac{1}{t}\int_{0}^{t} f(\theta,u)\ du, \ &\mbox{ if } x<1,\\
	0,\ &\mbox{ if } x\geq 1.
		\end{cases}$
	\end{lemma}

	\begin{proof}
		Using the definition of the kernel $H$ and grouping terms, it follows 
	\begin{align*}
	&\int_{0}^{2\pi}\hspace{-0.5em} \int_{0}^{1} T_{H}f(\theta,x)\ g(\theta,x)\ dx\ d\theta=\int_{0}^{2\pi} \hspace{-0.5em}\int_{0}^{2\pi} \hspace{-0.5em}\int_{0}^{1}\hspace{-0.5em}\int_{0}^{1} H(\theta,\varphi,x,y) f(\varphi,y) g(\theta,x)\ dx\ dy\ d\varphi\ d\theta\\
	&=\iiiint_{0\leq x,y\leq  |\theta-\varphi|\leq 1} \frac{1}{|\theta-\varphi|^2} f(\varphi,y) g(\theta,x)\ dx\ dy\ d\varphi\ d\theta\\
	&=\iint_{|\varphi-\theta|\leq 1} \left(\frac{1}{|\theta-\varphi|}\int_{0}^{|\theta-\varphi|} f(\varphi,y)\ dy\right)\left(\frac{1}{|\theta-\varphi|}\int_{0}^{|\theta-\varphi|} g(\varphi,x)\ dx\right)\ d\theta\ d\varphi\\
	&\leq \int_{0}^{2\pi} \hspace{-0.5em}\int_{0}^{2\pi} Rf(\theta, |\varphi-\theta|) Rg(\varphi, |\varphi-\theta|)\ d\theta\ d\varphi.
	\end{align*}
\end{proof}

	\begin{remark}\label{BoundRop}
		Let $1<p<\infty$. Notice that if $0\leq x\leq x_1\leq 1$ then $Rf(\theta, x)\geq Rf(\theta,x_1)$.
		Moreover, for $\theta \in \mathbb{T}$ fixed we define $f_{\theta}(x):=f(\theta,x)$. Therefore, since  $Rf(\theta,x)\leq M f_{\theta}(x)$, where $M$ is the Hardy-Littlewood maximal function,  there is a constant $C_{p}>0$ such that $\|Rf(\theta,\cdot)\|_{L^{p}[0,1]}\leq C_{p}\|f_{\theta}\|_{L^{p}[0,1]}$.
	\end{remark}

\begin{proof}[Proof of Theorem~\ref{BergmanProj}]
Bearing in mind the notation of the previous lemma, for $f\in {L^{q}(\mathbb{T},L^p[0,1])}$ and $g\in {L^{q'}(\mathbb{T},L^{p'}[0,1])}$ such that $\rho_{p,q}(f)\leq 1$ and $\rho_{p',q'}(g)\leq 1$ we consider the functions $F=Rf$ and $G=Rg$. Moreover, we define the following sequences of functions $f_{k}(\varphi)=F(\varphi,2^{-k})$ and $g_{k}(\varphi)=G(\varphi,2^{-k})$, $\varphi\in\mathbb{T}$ and $k\in \N$. Notice that for all $x\in I_{k}=[2^{-k},2^{-k+1})$ we have that $f_{k-1}(\varphi)\leq F(\varphi,x)\leq f_{k}(\varphi)$ and $g_{k-1}(\varphi)\leq G(\varphi,x)\leq g_{k}(\varphi)$. Indeed, it follows 
\begin{align*}
\sum_{k=1}^{\infty} f_{k-1}(\varphi) \chi_{I_k}(x)&\leq F(\varphi,x)\leq \sum_{k=1}^{\infty} f_{k}(\varphi) \chi_{I_k}(x),\\
\sum_{k=1}^{\infty} g_{k-1}(\varphi) \chi_{I_k}(x)&\leq G(\varphi,x)\leq \sum_{k=1}^{\infty} g_{k}(\varphi) \chi_{I_k}(x).
\end{align*}

Using Remark~\ref{BoundRop} and these inequalities, we obtain 
\begin{align*}
\int_{0}^{2\pi} \left(\sum_{k=1}^{\infty} f_{k-1}^{p}(\varphi)\ 2^{-k}\right)^{q/p}\ d\varphi\leq \int_{0}^{2\pi} \left(\int_{0}^{1} |F(\varphi,x)|^{p}\ dx\right)^{q/p}\ d\varphi \leq C_{p}^{q}
\end{align*}
and therefore
\begin{align}\label{bsequ}
\int_{0}^{2\pi} \left(\sum_{k=0}^{\infty}f_{k}^{p}(\varphi) 2^{-k}\right)^{q/p}\ d\varphi\leq 2^{q/p} C_{p}^{q}.
\end{align}
Following the same argument, we obtain the inequality for the sequence $\{g_k\}$.

Hence, by Lemma~\ref{PRineq} we have 
\begin{align*}
&\int_{0}^{2\pi} \hspace{-0.5em} \int_{0}^{1} (T_{H}f)\ g\ dx\ d\theta\leq \int_{0}^{2\pi} \hspace{-0.5em}\int_{0}^{2\pi} F(\theta, |\varphi-\theta|) G(\varphi, |\varphi-\theta|)\ d\theta\ d\varphi\\
&\leq \int_{0}^{2\pi} \left(\sum_{k=1}^{\infty} f_{k}(\theta) \int_{0}^{2\pi} g_{k}(\varphi) \chi_{I_k}(|\theta-\varphi|)\ d\varphi\right)d\theta\\
&\leq \int_{0}^{2\pi} \sum_{k=1}^{\infty} f_{k}(\theta)\ 2^{2-k}\left( \frac{1}{2^{2-k}}\int_{\theta-2^{-k+1}}^{\theta+2^{-k+1}} g_{k}(\varphi)\ d\varphi\right)\ d\theta\\
&\leq \int_{0}^{2\pi} \sum_{k=1}^{\infty} f_{k}(\theta)\ 2^{2-k} Mg_{k}(\theta)\ d\theta.
\end{align*}
Applying Hölder's inequality it follows 
\begin{align*}
&\int_{0}^{2\pi}\hspace{-0.5em} \int_{0}^{1} (T_{H})f(\theta,x)\ g(\theta,x)\ dx\ d\theta\leq 4 \int_{0}^{2\pi} \left(\sum_{k=1}^{\infty} f_{k}^{p}(\theta)\ 2^{-k}\right)^{1/p}\left(\sum_{k=1}^{\infty} Mg_{k}^{p'}(\theta)\ 2^{-k}\right)^{1/p'}\ d\theta\\
&\leq\left(\int_{0}^{2\pi }\left(\sum_{k=1}^{\infty} f_{k}^{p}(\theta)\ 2^{-k}\right)^{q/p}\ d\theta\right)^{1/q}\left(\int_{0}^{2\pi }\left(\sum_{k=1}^{\infty} Mg_{k}^{p'}(\theta)\ 2^{-k}\right)^{q'/p'}\ d\theta\right)^{1/q'}.
\end{align*}
Hence, by a classical result of Fefferman and Stein \cite[Theorem 1, p.107]{fefferman_stein_1971} and the inequalities \eqref{bsequ}, we get
\begin{align*}
\int_{0}^{2\pi} \hspace{-0.5em}\int_{0}^{1} (T_{H}f(\theta,x))\ g(\theta,x)\ dx\ d\theta&\leq 2^{2+1/p}\ C_{p}\ A_{p',q'} \left(\int_{0}^{2\pi }\left(\sum_{k=1}^{\infty} g_{k}^{p}(\theta)\ 2^{-k}\right)^{q'/p'}\ d\theta\right)^{1/q'}\\
&\leq 8\ C_{p}\ C_{p'}\ A_{p',q'}.
\end{align*}

Finally, we conclude the proof of the boundedness of the Bergman projection using  the last inequality with \cite[Theorem 1, p. 303]{benedek_panzone_1962}.
\end{proof}

An important consequence of this result is the following corollary about the dual of $RM(p,q)$ for $1<p,q<\infty$.

\begin{corollary}\label{dualopenbox}
	Let $1<p,q<\infty$. Then $(RM(p,q))^{\ast}\cong RM(p',q')$, where $\frac{1}{p}+\frac{1}{p'}=1$ and $\frac{1}{q}+\frac{1}{q'}=1$. 
\end{corollary}
\begin{proof}
	One part of the proof follows immediately. Indeed, applying the Hölder's inequality, one has that the functional defined by
	\begin{align*}
	\lambda_{g}(f)=\int_{\D} f(z)\overline{g(z)}\ dA(z), \quad f\in RM(p,q),\ g\in RM(p',q'),
	\end{align*}
	where $A$ is the Lebesgue measure on the unit disc $\D$,
	is bounded and $\|\lambda_g\|_{(RM(p,q))^{\ast}}\leq \rho_{p',q'}(g)$. Moreover it is unique, since if we assume that $\lambda_{g_1}=\lambda_{g_2}$ we have that $\lambda_{g_1}(z^n)=\lambda_{g_2}(z^n)$ for all $n\in\N$. Hence, $g_1=g_2$ because $\lambda_g(z^n)=\frac{\overline{a_n}}{n+1}$, where $a_n$ is the $n$-th Taylor coefficient of $g$.
	
	Now, let $\lambda$ be a functional in $(RM(p,q))^{\ast}$. We have to show that there exists $g\in RM(p',q')$ such that
	\begin{align*}
	\lambda(f)=\int_{\D} f(z)\overline{g(z)}\ dA(z) \quad \mbox{for every} \quad f\in RM(p,q).
	\end{align*}
	
	 Using the Hahn-Banach theorem, this functional can be extended to a certain  $\Lambda\in \left(L^q(\mathbb{T},L^p[0,1])\right)^{\ast}$ such that $\|\lambda\|_{(RM(p,q))^{\ast}}=\|\Lambda\|_{(L^q(\mathbb{T},L^p[0,1]))^{\ast}}$.
	Now, by means of \cite[Theorem 1, p. 304]{benedek_panzone_1962} there is a function $h\in L^{q'}(\mathbb{T},L^{p'}[0,1])$ such that
	\begin{align*}
	\Lambda(f)=\int_{\D} f(z)  \overline{h(z)}\ dA(z) \quad \mbox{for every} \quad f\in L^q(\mathbb{T},L^p[0,1])
	\end{align*}
	and $\|\Lambda\|_{(L^{q}(\mathbb{T},L^{p}[0,1]))^{\ast}}=\|h\|_{L^{q'}(\mathbb{T},L^{p'}[0,1])}$.
	
	 Let $g=T_{K}h$, where $T_K$ is the Bergman projection, and notice that, using Theorem~\ref{BergmanProj}, $g\in RM(p',q')$. So, by Fubini's theorem we have, for $f\in RM(p,q)$, 
	\begin{align*}
	\lambda(f)&=\Lambda(f)=\int_{\D} f(z)\overline{h(z)}\ dA(z)=\int_{\D}\int_{\D} \frac{f(w)}{(1-z\overline{w})^2}\ dA(w) \overline{h(z)}\ dA(z)\\
	&=\int_{\D} f(w)\int_{\D} \frac{\overline{h(z)}}{(1-z\overline{w})^2}\ dA(z) \ dm(w)=\int_{\D} f(w)\overline{T_{K}h(w)} \ dA(w)\\
	&=\int_{\D} f(w)\overline{g(w)} \ dA(w).
	\end{align*}
	 Also, we obtain that $\rho_{p,q}(g)\leq C \|h\|_{L^{q'}(\mathbb{T},L^{p'}[0,1])}=C \|\lambda\|_{(RM(p,q))^{\ast}}$ by Theorem~\ref{BergmanProj}.
\end{proof}

For the cases not covered by Theorem~\ref{BergmanProj}, its statement does not hold. In fact, we have 

\begin{theorem}\label{Pnobound}
	Let $1\leq p,q\leq +\infty$. If $\max\{p,q\}=+\infty$ or $\min\{p,q\}=1$, then the Bergman projection  $P$ does not send $L^q(\T,L^p[0,1]) $ into $ RM(p,q)$.
\end{theorem}

Before starting with the proof of the theorem, we state the following elementary lemma.

	\begin{lemma}	\label{lemmastolzreg}
	If $z,w\in\Omega:=\{re^{i\theta}:\, 0<\theta<1/2,\, 0<r<1-2\theta\}$, then 
	\begin{enumerate}
		\item $|1-z|\asymp 1-|z|$,
		\item $\left|\mathrm{Arg}\left(\frac{1-z}{1-w}\right)\right|\leq \arctan\left(\frac{1}{2}\right)<\frac{\pi}{4}$,
		\item $\mathrm{Re}\left(\frac{1-z}{1-w}\right)^2\geq \frac{3}{5} \left|\frac{1-z}{1-w}\right|^2$.
	\end{enumerate}
\end{lemma}
\begin{proof}
	The first identity follows immediately using the triangular inequality and the definition of the set $\Omega$:
	\begin{align*}
	1-r\leq |1-re^{i\theta}|\leq \sqrt{(1-r)^2+\theta^2}\leq \sqrt{\frac{5}{4}}(1-r).
	\end{align*} 
	To prove the second one it is enough to show that 
	$\tan(\textrm{Arg}(1-\overline{z}))\leq \frac{1}{2}$ for $z\in\Omega$, because we have that $\textrm{Arg}(1-\overline{z})\in(0,\arctan(1/2))$ and $\textrm{Arg}(1-{z})\in(-\arctan(1/2),0)$.
	Clearly, one can see, for $re^{i\theta}\in \Omega$, that
	\begin{align*}
	\tan(\textrm{Arg}(1-\overline{z}))=\frac{r\sin(\theta)}{1-r\cos(\theta)}\leq \frac{(1-2\theta)\sin(\theta)}{1-(1-2\theta)\cos(\theta)}.
	\end{align*}
	
	To finish the proof of (2), we have to show that $\frac{(1-2\theta)\sin(\theta)}{1-(1-2\theta)\cos(\theta)}\leq \frac{1}{2}$. But this is clear because the function $f(\theta)=\frac{1}{2}(1-(1-2\theta)\cos(\theta))-(1-2\theta)\sin(\theta)$ for $\theta\in \left(0,\frac{1}{2}\right)$ satisfies that $f(0)=0$ and  $f'(\theta)=2\theta\cos(\theta)+ \frac{1}{2}(5-2\theta)\sin(\theta)\geq 0$ for $\theta\in \left(0,\frac{1}{2}\right)$. 
	
	The last inequality follows immediately from (2).
\end{proof}

	\begin{proof}[Proof of Theorem~\ref{Pnobound}]
	\textbf{The case $p=+\infty$}.  Let us recall that the Bergman projection $P$ is a bounded operator from $L^\infty(\D)$ onto the Bloch space $\mathcal{B}$ (see \cite[p. 47, Theorem 7]{duren_schuster} or \cite[p. 102, Theorem 5.2]{zhu_2007}). Moreover, using lacunary sequences, it is possible to find functions in  $\mathcal{B}$ whose Taylor coefficients do not go to zero (see \cite[Lemma 2.1]{ACP}). Therefore,  $\mathcal{B}\nsubseteq H^q$, $1\leq q\leq +\infty$. Thus, Bergman projection $P$ is not bounded from $L^q(\T,L^\infty[0,1])$ to $RM(\infty,q)=H^{q}$.\newline
	
	\noindent \textbf{The case $q=+\infty$}.  We show that there exists a function $f\in L^{\infty}(\T,L^p[0,1])$ such that 
	\begin{align*}
	|P(f)(a)|\gtrsim (1-a)^{-1/p}, \text{ for every } a\in \left(\frac{3}{4}, 1\right),
	\end{align*}
	so that $P(f)\notin RM(p,\infty)$. To prove this, take the set $$\Omega=\left\{re^{i\theta}\ :\ 0<\theta<1/2, 0<r<1-2\theta \right\}.$$ Given $\alpha\in \R$, consider the function
	\begin{align*}
	f(re^{i\theta}):=\begin{cases}
	0, & re^{i\theta}\notin \Omega,\\
	\theta^{\alpha} K(1-\theta,re^{-i\theta}), & re^{i\theta}\in \Omega,
	\end{cases}
	\end{align*}
	where, as usual, $K$ is the Bergman kernel. Taking $\alpha=2-\frac{1}{p}=1+\frac{1}{p'}$, we have $f\in L^{\infty}(\T,L^p[0,1])$. Indeed, for $0<\theta<1/2$,
	\begin{align*}
&\int_{0}^{1} |f(re^{i\theta})|^{p}\ dr=\theta^{p\alpha} \int_{0}^{1-2\theta} |K(1-\theta,re^{-i\theta})|^{p}\ dr\leq \theta^{p\alpha} \int_{0}^{1} \frac{dr}{(1-(1-\theta)r)^{2p}}\\
&\qquad =\frac{\theta^{p\alpha} }{2p-1}\frac{\theta^{1-2p}-1}{1-\theta}\leq \frac{2}{2p-1}\theta^{p\alpha +1-2p}=\frac{2}{2p-1}<+\infty.
	\end{align*}
	
	Now let us see that this function $f$ satisfies that $|P(f)(a)|\gtrsim (1-a)^{-1/p}$ for every $a\in \left(\frac{3}{4},1\right)$. We have that the Bergman projection of the function $f$, for $a\in\left(0,1\right)$, is
	\begin{align*}
	P(f)(a)&=\int_{0}^{1/2} \theta^{\alpha} \left(\int_{0}^{1-2\theta} \frac{dr}{(1-are^{i\theta})^2(1-(1-\theta)re^{-i\theta})^2}\right)\ d\theta\\
	&=\int_{0}^{1/2} \theta^{\alpha} \left(\int_{0}^{1-2\theta} \left(\frac{1-(1-\theta)re^{i\theta}}{1-are^{i\theta}}\right)^2\frac{dr}{|1-(1-\theta)r e^{-i\theta}|^4}\right)\ d\theta.
	\end{align*}
	By Lemma~\ref{lemmastolzreg} (applying first $(3)$ and then $(1)$), we obtain
	\begin{align*}
	|P(f)(a)|&\geq  \mathrm{Re} \,[P(f)(a)]\geq \frac{3}{5} \int_{0}^{1/2} \theta^{\alpha} \left(\int_{0}^{1-2\theta} \left|\frac{1-(1-\theta)re^{i\theta}}{1-are^{i\theta}}
	\right|^2\frac{dr}{|1-(1-\theta)r e^{-i\theta}|^4}\right)\ d\theta\\
	&\asymp \int_{0}^{1/2} \theta^{\alpha} \left(\int_{0}^{1-2\theta} 
\frac{dr}{(1-ar)^2(1-(1-\theta)r)^2}\right)\ d\theta \\
&\geq \int_{0}^{1-a} \theta^{\alpha} \left(\int_{0}^{1-2\theta} 
\frac{dr}{(1-ar)^4}\right) d\theta =\frac{1}{3a} \int_{0}^{1-a} \theta^{\alpha} \left(
(1-a(1-2\theta))^{-3}-1\right) d\theta.
	\end{align*}
Using that $\theta <1-a$ and $3/4\leq a<1$ we deduce  $(1-a(1-2\theta))^{-3}-1\geq (1-a(1-2\theta))^{-3}/2$ and  $1-a(1-2\theta)<3(1-a)$. Hence
	\begin{align*}
\frac{1}{3a} \int_{0}^{1-a} \theta^{\alpha} \left(
(1-a(1-2\theta))^{-3}-1\right) d\theta  & \geq \frac{1}{6a} \int_{0}^{1-a} \theta^{\alpha} 
(1-a(1-2\theta))^{-3} d\theta\\
&\geq \frac{1}{6a}\frac{1}{27}\frac{1}{(1-a)^{3}}\int_{0}^{1-a}\theta^{\alpha}\, d\theta \\
&\asymp (1-a)^{\alpha-2}=(1-a)^{-1/p}.
\end{align*}
Thus, for $a>3/4$, 	we have $|P(f)(a)| \gtrsim   (1-a)^{-1/p}$ and  the function $P(f)$ does not belong to $RM(p,\infty)$. \newline
	
	
\noindent	\textbf{The remaining cases.} For the remaining cases, we use the fact that if the Bergman projection $P: L^{q'}(\T,L^{p'}[0,1])\to RM(p',q')$ is bounded then $P: L^{q}(\T,L^{p}[0,1])\to RM(p,q)$ is bounded since, for $f\in L^{q}(\T,L^{p}[0,1])$,
	\begin{align*}
	\rho_{p,q}{(P(f))}&\asymp \rho_{p,q}(r P(f)))=\sup_{g\in B_{ L^{q'}(\T,L^{p'}[0,1])}}\left|\int_{\D} P(f)(w)\ \overline{g(w)}\ dA(w)\right|\\
	&=\sup_{g\in B_{ L^{q'}(\T,L^{p'}[0,1])}}\left|\int_{\D} f(w) \ \overline{P(g)(w)}\ dA(w)\right| \\
	&\leq \rho_{p,q}(f)\sup_{g\in B_{ L^{q'}(\T,L^{p'}[0,1])}} \rho_{p',q'}(P(g))\leq C\rho_{p,q}(f) .
	\end{align*}
	where $C$ is the norm of the operator $P: L^{q'}(\T,L^{p'}[0,1])\to RM(p',q')$ and, as usual,  $B_{L^{q'}(\T,L^{p'}[0,1])}$ denotes the unit ball of $L^{q'}(\T,L^{p'}[0,1])$.
\end{proof}


\end{document}